\documentclass[10pt,leqno]{amsart}
\usepackage{amsfonts,amssymb}
\usepackage{color}
\usepackage{graphicx}
\usepackage{dsfont}
\usepackage{tikz-cd}

\newtheorem{theorem}{Theorem}[section]
\newtheorem{lemma}[theorem]{Lemma}
\newtheorem{corollary}[theorem]{Corollary}
\newtheorem{proposition}[theorem]{Proposition}
\newtheorem{assumption}[theorem]{Assumption}

\theoremstyle{definition}
\newtheorem{definition}[theorem]{Definition}

\newtheorem{remark}[theorem]{Remark}

\numberwithin{equation}{section}

\newcommand{\R}{{\mathbb R}}

\newcommand{\N}{{\mathbb N}}

\newcommand{\cW}{{\cal W}}

\renewcommand{\d}{{\mathrm d}}
\renewcommand{\u}{{\mathfrak u}}
\newcommand{\w}{{\mathfrak w}}

\newcommand{\dist}{\operatorname{dist}}

\renewcommand{\epsilon}{\varepsilon}
\newcommand{\cal}[1]{{\mathcal{#1}}}

\newcommand{\supp}{\operatorname{supp}}
\newcommand{\diam}{\operatorname{diam}}
\renewcommand{\i}{i}

\newcommand{\x}{{\mathrm x}}
\newcommand{\y}{{\mathrm y}}
\newcommand{\Hd}{\mathcal{H}_{d-1}}
\renewcommand{\H}{\mathcal{H}}

\newcommand{\eps}{\varepsilon}

\newcommand{\dd}{{\mathrm{d}}}

\begin{document}

\title[Hardy's inequality]{Hardy's inequality for functions vanishing on a part
of the boundary}

\author{Moritz Egert}
\address{Technische Universit\"at Darmstadt, Fachbereich
Mathematik, Schlossgartenstr.\@ 7, D-64298 Darmstadt, Germany}
\email{egert@mathematik.tu-darmstadt.de}

\author{Robert Haller-Dintelmann}
\address{Technische Universit\"at Darmstadt, Fachbereich
Mathematik, Schlossgartenstr.\@ 7, D-64298 Darmstadt, Germany}
\email{haller@mathematik.tu-darmstadt.de}

\author{Joachim Rehberg}
\address{Weierstrass Institute for Applied Analysis and Stochastics,
 Mohrenstr.\@ 39, D-10117 Berlin, Germany}
\email{rehberg@wias-berlin.de}

\subjclass[2010]{26D10, 26D15, 42B20, 42B37}
\date{\today}
\keywords{Hardy's inequality, uniform fatness, Poincar\'{e}'s inequality,
Sobolev extension operator}
\thanks{The first author was supported by ``Studienstiftung des deutschen Volkes''.}

\begin{abstract}
We develop a geometric framework for Hardy's inequality on a bounded domain
when the functions do vanish only on a closed portion of the boundary.
\end{abstract}

\maketitle

%
%
%
\section{Introduction}

\noindent Hardy's inequality is one of the classical items in analysis 
\cite{kufner, opic}. Two milestones among many others in the development of
the theory seem to be the result of Necas \cite{necas} that Hardy's inequality
holds on strongly Lipschitz domains  and the insight of Maz'ya \cite{maz1960},
\cite[Ch.~2.3]{mazsob} that its validity depends on measure theoretic conditions
on the domain. Rather recently, the geometric framework in which Hardy's
inequality remains valid was enlarged up to the frontiers of what is possible --
as long as the boundary condition is purely Dirichlet, see  \cite{korte, juha},
compare also \cite{Anc,lewis,Wannebo}. Moreover, over the last years it became
manifest that Hardy's inequality plays an eminent role in modern PDE theory,
see e.g.\@ \cite{caza, vandenBerg, rakoto, aermark, ekholm, dambrosio, garcia,
kang, lisk, marcus}.

What has not been treated systematically is the case where \emph{only a part}
$D$ of the boundary of the underlying domain $\Omega$ is involved, reflecting
the Dirichlet condition of the equation on this part -- while on $\partial
\Omega \setminus D$ other boundary conditions may be imposed, compare
\cite{ed/kuf,kovarik,aermark,korol,chechkin} including references therein. The
aim of
this paper is to set up a geometric framework for the domain $\Omega$ and the
Dirichlet boundary part $D$ that allow to deduce the corresponding Hardy
inequality
\begin{align*}
 \int_\Omega \left| \frac{u}{\dist_D} \right|^p \; \dd \mathrm{x} \le c
	\int_\Omega |\nabla u|^p \; \dd \mathrm{x}.
\end{align*}
As in the well established case $D = \partial \Omega$ we in
essence only require that $D$ is $l$-thick in the sense of \cite{juha}. In
our context this condition can be understood as an extremely weak compatibility
condition between $D$ and $\partial \Omega \setminus D$.

Our strategy of proof is first to reduce to the case $D = \partial \Omega$ by
purely topological means, provided two major tools
are applicable: An extension operator $\mathfrak E : W^{1,p}_D(\Omega) \to
W^{1,p}_D(\R^d)$, the subscript $D$ indicating the subspace of those Sobolev
functions which vanish on $D$ in an appropriate sense, and a Poincar\'{e}
inequality on $W^{1,p}_D(\Omega)$. This abstract result is established in
Section~\ref{sec-proof}. In a second step in Sections~\ref{extens} and
\ref{sec-poincare} these partly implicit conditions are substantiated by more
geometric assumptions that can be checked -- more or less -- by appearance. In
particular, we prove that under the mere assumption that $D$ is closed, every
linear continuous extension operator $W_D^{1,p}(\Omega) \to W^{1,p}(\R^d)$ that
is constructed by the usual procedure of gluing together local extension
operators preserves the Dirichlet condition on $D$. This result even carries
over to higher-order Sobolev spaces and sheds new light on some of the deep
results on Sobolev extension operators obtained in \cite{mitrea}. 

It is of course natural to ask, whether Hardy's inequality also
characterizes the space $W^{1,p}_D(\Omega)$, i.e.\ whether the latter is
precisely the space of those functions $u \in W^{1,p}(\Omega)$ for which
$u/\dist_D$ belongs to $L^p(\Omega)$. Under very mild geometric assumptions
we answer this question to the affirmative in Section~\ref{sec-conv-hardy}.

Finally, in Section~\ref{sec-general} we attend to the naive intuition that
the part of $\partial \Omega$ that is far away from $D$ should only be
circumstantial for the validity of Hardy's inequality and in fact we succeed to
weaken the previously discussed geometric assumptions considerably.

\section{Notation}
\noindent Throughout we work in Euclidean space $\R^d$, $d \geq 1$. We use
$\x$, $\y$, etc.\ for vectors in $\R^d$ and denote the open ball in $\R^d$
around $\x$ with radius $r$ by $B(\x, r)$. The letter $c$ is reserved for
generic constants that may change their value from occurrence to
occurrence. Given $F \subset \R^d$ we write $\dist_F$ for the function that
measures the distance to $F$ and $\diam(F)$ for the diameter of $F$.

In our main results on Hardy's inequality we denote the underlying domain and
its Dirichlet part by $\Omega$ and $D$. The various side results that are
interesting in themselves and drop off on the way are identified by the use of
$\Lambda$ and $E$ instead.

Next, let us introduce the common first-order Sobolev spaces of functions
`vanishing' on a part of the closure of the underlying domain that are most
essential for the formulation of Hardy's inequality.

\begin{definition}\label{d-WkpD}
If $\Lambda$ is an open subset of $\R^d$ and $E$ is a closed subset of
$\overline \Lambda$, then for $p \in {[1,\infty[}$ the space
$W^{1,p}_E(\Lambda)$ is defined as the completion of
 \begin{align*}
   C^\infty_E(\Lambda) := \{ v|_\Lambda : v \in C_0^\infty(\R^d),
	\; \supp(v) \cap E = \emptyset \}
 \end{align*}
with respect to the norm $v \mapsto \bigl( \int_\Lambda |\nabla v|^p +
|v|^p \; \d \x \bigr)^{1/p}$. More generally, for $k \in \N$ we define
$W^{k,p}_E(\Lambda)$ as the closure of $C^\infty_E(\Lambda)$ with respect to
the norm $v \mapsto \bigl( \int_\Lambda \sum_{j=0}^k |D^j v|^p \; \d \x
\bigr)^{1/p}$. 
\end{definition}

The situation we have in mind is of course when $\Lambda = \Omega$ and $E = D$
is the Dirichlet part $D$ of the boundary $\partial \Omega$. 

As usual, the Sobolev spaces $W^{k,p}(\Lambda)$ are defined as the
space of those
$L^p(\Lambda)$ functions whose distributional derivatives up to order $k$ are in
$L^p(\Lambda)$, equipped with the natural norm. Note that by definition
$W^{k,p}_0(\Lambda) = W^{k,p}_{\partial \Lambda}(\Lambda)$ but in general
$W^{k,p}_\emptyset(\Lambda) \subsetneq W^{k,p}(\Lambda)$,
cf.\ \cite[Sec.~1.1.6]{mazsob}

\section{Main results} \label{sec-hart}
\noindent The following version of Hardy's inequality for functions vanishing on
a part of the boundary is our main result. Readers not familiar with the
measure theoretic concepts used to describe the regularity of the Dirichlet
part $D$ may refer to Section~\ref{subsec-concepts for D}
beforehand.

\begin{theorem} \label{t-hardy}
 Let $\Omega \subset \R^d$ be a bounded domain, $D \subset \partial \Omega$ be a
 closed part of the boundary and $p \in {]1,\infty[}$. Suppose that the
following
 three conditions are satisfied.
 \begin{enumerate} 
 \item \label{t-hardy:i} The set $D$ is $l$-thick for some $l \in
{]d-p,d]}$.
 \item \label{t-hardy:ii} The space $W^{1,p}_D(\Omega)$ can be equivalently
	normed by $\|\nabla \cdot \|_{L^p(\Omega)}$.
 \item \label{t-hardy:iii} There is a linear continuous extension operator
$\mathfrak E :
	W^{1,p}_D(\Omega) \to W^{1,p}_D(\R^d)$.
 \end{enumerate}
 Then there is a constant $c>0$ such that Hardy's inequality
 \begin{align}\label{e-hardy0}
  \int_\Omega \left| \frac{u}{\dist_D} \right|^p \; \dd \mathrm{x} \le c
	\int_\Omega |\nabla u|^p \; \dd \mathrm{x}
 \end{align}
 holds for all $u \in W^{1,p}_D(\Omega)$.
\end{theorem}

Of course the conditions \ref{t-hardy:ii} and \ref{t-hardy:iii} in
Theorem~\ref{t-hardy} are rather abstract and should be supported by more
geometrical ones. This will be the content of Sections~\ref{extens} and
\ref{sec-poincare} where we shall give an extensive kit of such conditions. In
particular, we will obtain the following version of
Hardy's inequality. 

\begin{theorem}[A special Hardy inequality] \label{t-hardy concrete}
 Let $\Omega \subset \R^d$ be a bounded domain and $p \in {]1,\infty[}$. Let $D
 \subset \partial \Omega$ be $l$-thick for some $l \in {]d-p,d]}$ and assume
 that for every $\x \in \overline{\partial\Omega \setminus D}$ there is an open
 neighborhood $U_\x$ of $\x$ such that $\Omega \cap U_\mathrm x$ is a
 $W^{1,p}$-extension domain. Then there is a constant $c >0$ such that
 \begin{align*}
  \int_\Omega \left| \frac{u}{\dist_D} \right|^p \; \dd \mathrm{x} \le c
	\int_\Omega |\nabla u|^p \; \dd \mathrm{x}, \quad u \in
 W_D^{1,p}(\Omega).
 \end{align*}
\end{theorem}

\begin{remark}
The assumptions in the above theorem are met for all $p \in {]1,\infty[}$ if $D$
is a $(d-1)$-set and for every $\x \in \overline{\partial\Omega \setminus D}$
there is an open neighborhood $U_\x$ of $\x$ such that $\Omega \cap U_\mathrm x$
is a Lipschitz- or more generally an $(\eps,\delta)$-domain, see
Subsections~\ref{subsec-concepts for D} and \ref{Subsec: geometric
conditions} for definitions.
\end{remark}

Still, as we believe, the abstract framework traced out by the second and the
third condition of Theorem~\ref{t-hardy} has the advantage that other sufficient
geometric conditions for Hardy's inequality -- tailor-suited for future
applications -- can be found much more easily. In fact the second condition is
equivalent to the validity of Poincar\'e's inequality
 \begin{align*}
  \|u\|_{L^p} \le c \|\nabla u\|_{L^p(\Omega)},
  \quad u \in W^{1,p}_D(\Omega),
 \end{align*}
that is clearly \emph{necessary} for Hardy's inequality \eqref{e-hardy0}. We
give a detailed discussion of Poincar\'e's inequality within the
present context in Section~\ref{sec-poincare}. For further reference the reader
may consult \cite[Ch.\ 4]{ziemer}. Concerning the third condition note carefully
that we require the extension operator to preserve the Dirichlet boundary
condition on $D$. Whereas extension of Sobolev functions is a well-established
business, the preservation of traces is much more delicate and we devote
Subsection~\ref{Subsec: preservation of traces} to this problem.

It is interesting to remark that under geometric assumptions very similar to
those in
Theorem~\ref{t-hardy concrete} the space $W_D^{1,p}(\Omega)$ is the largest
subspace of $W^{1,p}(\Omega)$ in which Hardy's inequality can hold. This is
made precise by our third main result.

\begin{theorem} \label{t-converse hardy}
Let $\Omega \subset \R^d$ be a bounded domain and $p \in {]1,\infty[}$. Let $D
\subset \partial \Omega$ be porous and $l$-thick for some $l \in {]d-p,d]}$.
Finally assume that for every $\x \in \overline{\partial\Omega \setminus D}$
there is an open neighborhood $U_\x$ of $\x$ such that $\Omega \cap U_\mathrm x$
is a $W^{1,p}$-extension domain. If $u \in W^{1,p}(\Omega)$ is such that
$u/\dist_D \in L^p(\Omega)$, then already $u \in W_D^{1,p}(\Omega)$.
\end{theorem}

\begin{remark} \label{r-converse hardy}
\begin{enumerate}
 \item The assumption on $D$ are met if $D$ is an $l$-set for some $l \in
{]d-p,d[}$, see Remark~\ref{r-porous} below. 
 \item In the case $D = \partial \Omega$ the conclusion of
 Theorem~\ref{t-converse hardy} is classical \cite[Thm.~V.3.4]{ed/ev} and
 remains true without any assumptions on $\partial \Omega$.
\end{enumerate}
\end{remark}

In Section~\ref{sec-proof} we give the proof of the general Hardy inequality
from Theorem~\ref{t-hardy}. The proofs of Theorem~\ref{t-hardy concrete} and
\ref{t-converse hardy} are postponed to the end of Sections~\ref{extens} and
\ref{sec-conv-hardy}, respectively.

\section{Preliminaries} \label{sec-preliminaries}
\subsection{Regularity concepts for the Dirichlet part}\label{subsec-concepts
for D}

For convenience we recall the notions from geometric measure theory that are
used to describe the regularity of the Dirichlet part $D$ in Hardy's inequality.
For $l\in {]0,\infty[}$ the $l$-dimensional Hausdorff measure of $F \subset
\R^d$ is
\begin{align*}
 \mathcal H_l(F):= \liminf_{\delta \to 0} \Big \{ \sum_{j= 1}^\infty
 \diam(F_j)^l : F_j \subset \R^d,\, \diam(F_j) \leq \delta, \, F \subset
 \bigcup_{j=1}^\infty F_j \Big \}
\end{align*}
and its centered Hausdorff content is defined by
\begin{align*}
 \mathcal H_l^\infty(F):= \inf \Big \{ \sum_{j= 1}^\infty r_j^l : \x_j 
 \in F, \, r_j > 0, \, F \subset \bigcup_{j=1}^\infty B(\x_j, r_j) \Big \}.
\end{align*}

\begin{definition} \label{d-lthick}
 Let $l \in {]0,\infty[}$. A non-empty compact set $F \subset \R^d$ is called
 \emph{$l$-thick} if there exist $R > 0$ and $\gamma > 0$ such
 that 
 \begin{equation} \label{e-lthick}
 \mathcal H^\infty_l (F \cap B(\x,r) ) \ge \gamma\, r^{l}
 \end{equation}
 holds for all $\x \in F$ and all $r \in {]0,R]}$. It is called \emph{$l$-set}
 if there are two constants $c_0, c_1 >0$ such that
 \begin{align*}
  c_0 r^l \le \mathcal H_l(F \cap B(\x,r)) \le c_1 r^l
 \end{align*}
holds for all $\x \in F$ and all $r \in {]0,1]}$.
\end{definition}

\begin{remark} \label{r-othconst}
\begin{enumerate}
 \item  If \eqref{e-lthick} holds for constants $R$, $\gamma$, then for all
 $S \ge R$ it also holds with $R$ and $\gamma$ replaced by $S$ and $ \gamma R^l
 S^{-l}$, respectively. For more information on this notion of $l$-thick sets
 the reader can refer to \cite{juha}.
 \item The notion of $l$-sets is due to \cite[Sec.~II.1]{jons}. It can be
 extended literally to arbitrary Borel sets $F$, see \cite[Sec.~VII.1.1]{jons}.
\end{enumerate}
\end{remark}

\begin{definition} \label{d-porous}
A set $F \subset \R^d$ is \emph{porous} if for some $\kappa \leq 1$ the
following statement is true: For every ball $B(\x,r)$ with $\x \in \R^d$ and $0
< r \leq 1$ there is $\y \in B(\x,r)$ such that $B(\y, \kappa r) \cap F =
\emptyset$.
\end{definition}

\begin{remark} \label{r-porous}
It is known that a set $F \subset \R^d$ is porous if and only if its so-called
Assouad dimension is strictly less than the space dimension $d$, see
\cite[Thm.~5.2]{luukkainen}. Recently it was shown in \cite{juha3} that this
notion of dimension coincides with the one introduced by Aikawa, that is the
infimum of all $t>0$ for which there exists $c_t > 0$ such that
\begin{align*}
 \int_{B(\x,r)} \dist(\x,F)^{t-d} \; \d \x \leq c_t r^t, \quad \x \in F, r> 0.
\end{align*}
In particular, each $l$-set, $l \in {]0,d[}$, has Aikawa dimension equal to
$l$ and thus is porous \cite[Lem.~2.1]{juha2}.
\end{remark}

For a later use we include a proof of the following two elementary facts. We
remark that the first lemma is also implicit in \cite[Lem.~2]{carlmaz}.

\begin{lemma}
\label{l-lsets with hausdorff content}
Let $l \in {]0,\infty[}$. If $F \subset \R^d$ is a compact $l$-set, then there
are constants $c_0, c_1 > 0$ such that
\begin{align*}
c_0 r^l \le \H^\infty_l(F \cap B(\x,r)) \le c_1 r^l
\end{align*}
holds for all $r \in {]0,1[}$ and all $\x \in F$. In particular, $F$ is
$l$-thick.
\end{lemma}

\begin{proof}
We prove $\H^\infty_l(A) \leq \H_l(A) \leq c
\H^\infty_l(A)$ for all non-empty Borel subsets $A \subset F$. 

First, fix $\eps > 0$ and let $\{A_j\}_{j \in \N}$ be a covering of $A$ by sets
with diameter at most $\eps$. If $A_j \cap A \neq \emptyset$, then $A_j$ is
contained in an open ball $B_j$ centered in $A$ and radius such that $r_j^l =
\diam(A_j)^l + \eps 2^{-j}$. The so-obtained countable covering $\{B_j\}$ of
$A$ satisfies
\begin{align*}
 \sum_{\substack{j \in \N \\ A_j \cap A \neq \emptyset}} \diam(A_j)^l
\geq \sum_{\substack{j \in \N \\ A_j \cap A \neq \emptyset}} (r_j^l - \eps
2^{-j}) \geq \H^\infty_l(A) - \eps.
\end{align*}
Taking the infimum over all such coverings $\{A_j\}_{j \in \N}$ and passing to
the limit $\eps \to 0$ afterwards, $\H^\infty_l(A) \leq
\H_l(A)$
follows. Conversely, let $\{B_j\}_{j \in \N}$ be a covering of $A$ by
open balls with radii $r_j$ centered in $A$. If $r_j \leq 1$, then
$\H_l(F \cap B_j) \leq c r_j^l$ since by assumption $F$ is an
$l$-set, and if $r_j > 1$, then certainly $\H_l(F \cap B_j)
\leq \H_l(F) r_j^l$. Note carefully that $0 < \H_l(F) <
\infty$ holds for $F$ can be covered by finitely many balls with radius $1$
centered in $F$. Altogether,
\begin{align*}
 \sum_{j= 1}^\infty r_j^l 
\geq c \sum_{j= 1}^\infty \H_l(F \cap B_j) 
\geq c\H_l \Big(F \cap \bigcup_{j=1}^\infty B_j \Big) \geq
c\H_l(A).
\end{align*}
Passing to the infimum, $\H^\infty_l(A) \geq c \H_l(A)$
follows.
\end{proof}

\begin{lemma}
\label{l-thickness for smaller values}
If $F \subset \R^d$ is $l$-thick, then it is $m$-thick for every $m \in
{]0,l[}$.
\end{lemma}

\begin{proof}
Inspecting the definition of thick sets, the claim turns out to be a direct
consequence of the inequality
\begin{align*}
 \sum_{j=1}^N r_j^m
\geq  \Big(\sum_{j=1}^N r_j^l \Big)^{m/l} 
\end{align*}
for positive real numbers $r_1,\ldots r_N$.
\end{proof}

\subsection{Quasieverywhere defined functions}

The results of Sections~\ref{extens}-\ref{sec-conv-hardy} rely on deep
insights from potential theory and we shall recall the necessary notions
beforehand. For further background we refer e.g.\ to \cite{ad/he}.

\begin{definition}
\label{d-Besselcapacity}
Let $\alpha >0$, $p \in {]1, \infty [}$ and let $F \subset \R^d$. Denote by
$G_\alpha:=\mathcal{F}^{-1}((1+|\xi|^2)^{-\alpha/2})$ the Bessel kernel of order
$\alpha$. Then
\begin{align*}
 C_{\alpha,p}(F):= \inf \Big\{ \int_{\R^d} |f|^p : \text{$f \geq 0$ on $\R^d$
and $G_\alpha \ast f \geq 1$ on $F$} \Big\}
\end{align*}
is called \emph{$(\alpha,p)$-capacity} of $F$. The corresponding \emph{Bessel
potential space} is
\begin{align*}
 H^{\alpha,p}(\R^d): = \{G_\alpha \ast f : f \in L^p(\R^d)\} \quad \text{with
norm} \quad \|G_\alpha \ast f\|_{H^{\alpha,p}(\R^d)} = \|f\|_p.
\end{align*}
\end{definition}

It is well-known that for $k \in \N$ the spaces $H^{k,p}(\R^d)$ and
$W^{k,p}(\R^d)$ coincide up to equivalent norms \cite[Sec.~2.3.3]{triebel}. The
capacities $C_{\alpha,p}$ are outer measures on $\R^d$
\cite[Sec.~2.3]{ad/he}. A property that holds true for all $\x$ in some set $E
\subset \R^d$ but those belonging to an exceptional set $F \subset E$ with
$C_{\alpha,p}(F) = 0$ is said to be true \emph{$(\alpha,p)$-quasieverywhere} on
$E$, abbreviated \emph{$(\alpha,p)$-q.e}. A property that holds true
$(\alpha,p)$-q.e.\ also holds true $(\beta,p)$-q.e.\ if $\beta<\alpha$. This is
an easy consequence of \cite[Prop.~2.3.13]{ad/he}. A more involved result in
this direction is the following \cite[Thm.~5.5.1]{ad/he}

\begin{lemma}\label{l-transformation of capacities}
Let $\alpha, \beta > 0$ and $1 < p,q < \infty$ be such that $\beta q <
\alpha p < d$. Then each $C_{\alpha,p}$-nullset also is a $C_{\beta,q}$-nullset 
\end{lemma}

There is also a close connection between capacities and Hausdorff measures, cf.\
\cite[Ch.~5.]{ad/he} for an exhaustive discussion. Most important for us is
the following comparison theorem. In the case $p \in {]1, d]}$ this is proved
in \cite[Sec.~5]{ad/he} and if $p \in {]d,\infty[}$, then the result follows
directly from \cite[Prop.~2.6.1]{ad/he}.

\begin{theorem}[Comparison Theorem]\label{t-comparison}
Let $1 < p < \infty$ and suppose $\alpha, l > 0$ are such that $d-l < \alpha p
<\infty$. Then every $C_{\alpha,p}$-nullset is also a $\H_l$- and thus a
$\H_l^\infty$-nullset.
\end{theorem}

Bessel capacities naturally occur when studying convergence of average
integrals for Sobolev functions. In fact, if $\alpha > 0$, $p \in
{]1,\frac{d}{\alpha}]}$ and $u \in H^{\alpha,p}(\R^d)$, then
$(\alpha,p)$-quasievery $\y \in \R^d$ is a Lebesgue point for $u$ in the
$L^p$-sense, that is
\begin{align}
\label{e-pointwi}
	  \lim_{r \to 0} \frac {1}{|B(\y,r)|}\int_{B(\y,r)}
		u(\x) \; \d \x =: \u(\y)
\end{align}
and
\begin{align}
\label{e-Lebesgue}
	  \lim_{r \to 0} \frac {1}{|B(\y,r)|}\int_{B(\y,r)}
		|u(\x) - \u(\y)|^p \; \d \x = 0
\end{align}
hold \cite[Thm.\ 6.2.1]{ad/he}. The $(\alpha,p)$-quasieverywhere defined
function $\u$ reproduces $u$ within its $H^{\alpha,p}$-class. It gives rise to a
meaningful $(\alpha,p)$-quasieverywhere defined restriction $u|_E:= \u|_E$ of
$u$ to $E$ whenever $E$ has non-vanishing $(\alpha,p)$-capacity. For convenience
we agree upon that $u|_E = 0$ is true for all $u \in H^{\alpha,p}(\R^d)$ if $E$
has zero $(\alpha,p)$-capacity. Note also that these results remain true if $p
\in {]\frac{d}{\alpha},\infty[}$, since in this case $u$ has a H\"older
continuous representative $\u$ which then satisfies \eqref{e-pointwi} and
\eqref{e-Lebesgue} for every $\y \in \R^d$.

We obtain an alternate definition for Sobolev spaces with partially vanishing
traces.

\begin{definition}
\label{d-Sob cap space}
Let $k \in \N$, $p \in {]1, \infty [}$ and $E \subseteq \R^d$ be closed. Define
\begin{align*}
 \cW_E^{k,p}(\R^d):= \big\{u \in W^{k,p}(\R^d): &\text{ $D^\beta u|_E = 0$
 holds $(k-|\beta|,p)$-q.e.\ on $E$} \\ &\text{ for all multiindices $\beta$, $0
\leq |\beta| \leq k-1$} \big\}
\end{align*}
and equip it with the $W^{k,p}(\R^d)$-norm.
\end{definition}

The following theorem of Hedberg and Wolff is also known as the
\emph{$(k,p)$-synthesis}.

\begin{theorem}[{\cite[Thm.\ 9.1.3]{ad/he}}] \label{t-coincid}
The spaces $W_E^{k,p}(\R^d)$ and $\cW_E^{k,p}(\R^d)$ coincide whenever $k \in
\N$, $p \in {]1,\infty[}$ and $E \subset \R^d$ is closed.
\end{theorem}

Hedberg and Wolff's theorem manifests the use of capacities in the study of
traces of Sobolev functions. However, if one invests more on the
geometry of $E$, e.g.\ if one assumes that it is an $l$-set, then by the
subsequent recent result of Brewster, Mitrea, Mitrea and Mitrea capacities can
be replaced by the $l$-dimensional Hausdorff measure at each occurrence.

\begin{theorem}[{\cite[Thm.\ 4.4, Cor.\ 4.5]{mitrea}}]
\label{t-coincid-mitrea}
Let $k \in \N$, $p \in {]1,\infty[}$ and let $E \subset \R^d$ be closed and
additionally an $l$-set for some $l \in {]d-p,d]}$. Then
\begin{align*}
 W_E^{k,p}(\R^d) = \cW_E^{k,p}(\R^d)= \big\{u \in W^{k,p}(\R^d): &\text{
 $D^\beta u|_E = 0$ holds $\Hd$-a.e.\ on $E$} \\ &\text{ for all multiindices
 $\beta$, $0\leq |\beta| \leq k-1$} \big\},
\end{align*}
where on the right-hand side $D^\beta u|_E = 0$ means, as before, that for
$\Hd$-almost every $\y \in E$ the average integrals
$\frac{1}{|B(\y,r)|}\int_{B(\y,r)} D^\beta u(\x)  \; \d \x $ vanish in the limit
$r \to 0$.
\end{theorem}

%
%
%
\section{Proof of Theorem~\ref{t-hardy}} \label{sec-proof} We will deduce
Theorem \ref{t-hardy} from the following proposition that states
the assertion in the case $D = \partial \Omega$.

\begin{proposition}[\cite{juha}, see also \cite{korte}] \label{p-juha}
 Let $\Omega_\bullet \subseteq \R^d$ be a bounded domain and let $p \in
{]1,\infty[}$. If $\partial \Omega_\bullet$ is $l$-thick for some $l \in
{]d-p,d]}$, then Hardy's inequality is satisfied for all $u
\in W^{1,p}_0(\Omega_\bullet)$, i.e.\@ \eqref{e-hardy0} holds with $\Omega$
replaced by $\Omega_\bullet$ and $D$ by $\partial \Omega_\bullet$.
\end{proposition}

Below we will reduce to the case $D = \partial \Omega$ by purely topological
means, so that we can apply Proposition~\ref{p-juha} afterwards. We will
repeatedly use the following topological fact.
\begin{align} \label{blacksquare}
\tag{$\blacksquare$} 
\begin{minipage}{0.92\textwidth}{Let $\{M_\lambda\}_{\lambda }$
be a family of connected subsets of a topological space. If $\bigcap_{\lambda }
M_\lambda \neq \emptyset$, then $\bigcup_{\lambda } M_\lambda$ is again
connected.} 
\end{minipage}
\end{align}
As required in Theorem~\ref{t-hardy} let now $\Omega \subseteq \R^d$ be a
bounded domain and let $D$ be a closed part of $\partial \Omega$. Then choose an
open ball $B \supseteq \overline \Omega$ that, in what follows, will be
considered as the relevant topological space. Consider
\begin{align*}
 \mathcal C := \{ M \subset B \setminus D : M \text{ open, connected
	and } \Omega \subset M \}
\end{align*}
and for the rest of the proof put
\begin{align*}
 \Omega_\bullet :=
	\bigcup_{M \in \mathcal C} M.
\end{align*}
In the subsequent lemma we collect some properties of $\Omega_\bullet$.
Our proof here is not the shortest possible, cf.\ \cite[Lem.\
6.4]{auschhallreh} but it has, however, the advantage to give a description
of $\Omega_\bullet$ as the union of $\Omega$, the boundary part $\partial
\Omega \setminus D$ and those connected components of $B \setminus \overline
\Omega$ whose boundary does not consist only of points from $D$. This completely
reflects the naive geometric intuition.

\begin{lemma} \label{l-1}
It holds $\Omega \subseteq \Omega_\bullet \subseteq B$. Moreover,
$\Omega_\bullet$ is open and connected and $\partial \Omega_\bullet = D$ in $B$.
\end{lemma}

\begin{proof}
 The first assertion is obvious. By construction $\Omega_\bullet$ is open.
 Since all elements from $\mathcal C$ contain $\Omega$ the connectedness of
 $\Omega_\bullet$ follows by \eqref{blacksquare}. It
 remains to show $\partial \Omega_\bullet = D$.

 Let $\x \in D$. Then $\x$ is an accumulation point of $\Omega$
 and, since $\Omega \subseteq \Omega_\bullet$, also of $\Omega_\bullet$. On the
 other hand, $\x \not \in \Omega_\bullet$ by construction. This implies
 $\x \in \partial \Omega_\bullet$ and so $D \subseteq \partial
 \Omega_\bullet$.

 In order to show the inverse inclusion, we first show that points from
 $\partial \Omega \setminus D$ cannot belong to $\partial \Omega_\bullet$.
 Indeed, since $D$ is closed, for $\x \in \partial \Omega \setminus D$
 there is a ball $B_\x \subseteq B$ around $\x$ that does not
 intersect $D$. Since $\x$ is a boundary point of $\Omega$, we have
 $B_\x \cap \Omega \neq \emptyset$. Both $\Omega$ and $B_\x$ are
 connected, so \eqref{blacksquare} yields that $\Omega \cup B_\x$ is
 connected. Moreover, this set is open, contains $\Omega$ and avoids $D$, so it
 belongs to $\mathcal C$ and we obtain $\Omega \cup B_\x \subseteq
 \Omega_\bullet$. This in particular yields $\x \in \Omega_\bullet$,
 so $\x \notin \partial \Omega_\bullet$ since $\Omega_\bullet$ is open.

 Summing up, we already know that $\x \in \overline \Omega$
 belongs to $\partial \Omega_\bullet$ if and only if $\x \in D$. So, it
 remains to make sure that no point from $B \setminus \overline \Omega$ belongs
 to $\partial \Omega_\bullet$.

 As $B \setminus \overline{\Omega}$ is open, it splits up into its open
 connected components $Z_0, Z_1, Z_2, \ldots$. There are possibly only finitely
 many such components but at least one. We will show in a first step that for
 all these components it holds $\partial Z_j \subseteq \partial \Omega$. This
 allows to distinguish the two cases $\partial Z_j \subseteq D$ and $\partial
 Z_j \cap (\partial \Omega \setminus D) \neq \emptyset$. In Steps 2 and 3 we
 will then complete the proof by showing that in both cases $Z_j$ does not
intersect $\partial \Omega_\bullet$.

\subsection*{Step 1: $\partial Z_j \subseteq \partial \Omega$ for all
$j$.}
 First note that $\partial Z_j \cap \Omega = \emptyset$
 for all $j$. Indeed, assuming this set to be non-empty and investing that
 $\Omega$ is open, we find that the set $Z_j \cap \Omega$ cannot be empty
 either and this contradicts the definition of $Z_j$.
 
 Now, to prove the claim of Step 1, assume by contradiction that, for some $j$,
 there is a point $\x \in \partial Z_j$ that does not belong to
 $\partial \Omega$. By the observation above we then have $\x \notin
 \overline \Omega$ and consequently there is a ball $B_\x$ around
 $\x$ that does not intersect $\overline \Omega$. Now, the set
 $B_\x \cup Z_j$ is connected thanks to \eqref{blacksquare}, avoids
 $\overline \Omega$ and includes $Z_j$ properly. However, this contradicts the
 property of $Z_j$ to be a connected component of $B \setminus \overline
 \Omega$.

\subsection*{Step 2: If $\partial Z_j \subseteq D$, then
 $\overline{\Omega}_\bullet \cap Z_j = \emptyset$.}
 We first note that it suffices to show $\Omega_\bullet \cap Z_j = \emptyset$.
 In fact, due to $\overline \Omega_\bullet = \partial \Omega_\bullet \cup
 \Omega_\bullet$ we then get $\overline{\Omega}_\bullet \cap Z_j =
 \emptyset$ since $Z_j$ is open.

 So, let us assume there is some $\x \in \Omega_\bullet \cap Z_j$. Then
 $\Omega_\bullet \cup Z_j$ is connected due to \eqref{blacksquare}. By
 assumption we have $\partial Z_j \subseteq D$ and by construction the sets
 $Z_j$ and $\Omega_\bullet$ are both disjoint to $D$. So we can infer that
 $\partial Z_j \cap (\Omega_\bullet \cup Z_j) = \emptyset$ and this allows us
 to write
 \[ \Omega_\bullet \cup Z_j = \bigl( \Omega_\bullet \cup Z_j \bigr) \cap \bigl(
	Z_j \cup ( B \setminus \overline Z_j ) \bigr) = Z_j \cup
	\bigl( \Omega_\bullet \cap ( B \setminus \overline Z_j ) \bigr).
 \]
 This is a decomposition of $\Omega_\bullet \cup Z_j$ into two open and
 mutually disjoint sets, so if we can show that both are nonempty then this
 yields a contradiction to the connectedness of $\Omega_\bullet \cup Z_j$ and
 the claim of Step 2 follows. Indeed, we even find
 \[ \Omega_\bullet \cap ( B \setminus \overline Z_j ) = \Omega_\bullet
	\setminus \overline Z_j = \Omega_\bullet \setminus (\partial Z_j \cup
	Z_j) \supset \Omega \setminus (D \cup Z_j) = \Omega \neq \emptyset,
 \]
 since both $D$ and $Z_j$ do not intersect $\Omega$.

\subsection*{Step 3: If $\partial Z_j \cap (\partial \Omega \setminus D) \neq
 \emptyset$, then $Z_j \subseteq \Omega_\bullet$.}
 Let $\x \in \partial Z_j \cap (\partial \Omega \setminus D)$, and let
 $B_\x$ be a ball around $\x$ that does not intersect $D$. The
 point $\x$ is a boundary point of $Z_j$, so $B_\x \cap Z_j \neq
 \emptyset$ and we obtain that $B_\x \cup Z_j$ is connected by
 \eqref{blacksquare}.
 By the same argument, also the set $B_\x \cup
 \Omega$ is connected and putting these two together a third reiteration of the
 argument yields that $(B_\x \cup \Omega) \cup (B_\x \cup Z_j) =
 \Omega \cup B_\x \cup Z_j$ is again connected.  This last set is
 open and does not intersect $D$, so it belongs to $\mathcal C$ and we end up
 with $\Omega \cup B_\x \cup Z_j \subseteq \Omega_\bullet$. In
 particular we have $Z_j \subseteq \Omega_\bullet$.
\end{proof}

\begin{remark} \label{r-unique}
 Conversely, it can be shown that the asserted properties characterize
 $\Omega_\bullet$ uniquely in the sense that if an open, connected subset $\Xi
 \supset \Omega$ of $B$ additionally satisfies $\partial \Xi = D$, then
 necessarily $\Xi = \Omega_\bullet$. In fact, since $\Xi \cap D = \emptyset$ one
 has $\Xi \subset \Omega_\bullet$, due to the definition of $\Omega_\bullet$. In
 order to obtain the inverse inclusion we write
 \begin{equation} \label{e-equa}
  \Omega_\bullet = \bigl( \Omega_\bullet \cap \Xi \bigr) \cup \bigl(
	\Omega_\bullet \cap \partial \Xi \bigr) \cup \bigl( \Omega_\bullet \cap
	(B \setminus \overline \Xi) \bigr) = \Xi \cup \bigl( \Omega_\bullet
	\cap (B \setminus \overline \Xi) \bigr),
 \end{equation}
 since $\Omega_\bullet \cap \partial \Xi = \Omega_\bullet \cap D = \emptyset$.
 Both $\Xi = \Xi \cap \Omega_\bullet$ and $\Omega_\bullet \cap (B \setminus
 \overline \Xi)$ are open in $\Omega_\bullet$, and $\Xi \supset \Omega$ is
 non-empty. Since $\Omega_\bullet$ is connected and $\Xi = \Xi \cap
 \Omega_\bullet$ is clearly disjoint to $\Omega_\bullet \cap (B \setminus
 \overline \Xi)$, this latter set must be empty. Thus, \eqref{e-equa} gives
 $\Xi =\Omega_\bullet$.
\end{remark}

\begin{corollary} \label{c-002HE}
 Consider $\Omega_\bullet$ as a subset of $\R^d$. Then $\Omega_\bullet$
 is open and connected. Moreover, either $\partial \Omega_\bullet = D$ or
 $\partial \Omega_\bullet = D \cup \partial B$.
\end{corollary}
\begin{proof}
 It is clear that $\Omega_\bullet$ remains open. Assume that $\Omega_\bullet$
 is not connected. Then there are disjoint open sets $U,V \subseteq \R^d$
 such that $\Omega_\bullet = U \cup V$. However, the property $\Omega_\bullet
 \subseteq B$ then gives $\Omega_\bullet = \Omega_\bullet \cap B = (U \cap B)
 \cup (V \cap B)$, where $U \cap B$ and $V \cap B$ are open in $B$ and disjoint
 to each other. This contradicts Lemma~\ref{l-1}.

 For the last assertion consider an annulus $A \subseteq B$ that is
 adjacent to $\partial B$ and does not intersect $\overline \Omega$. Let $Z_j$
 be the connected component of $B \setminus \overline{\Omega}$ that contains
 $A$. We distinguish again the two cases of Step 2 and Step 3 in the proof of
 Lemma~\ref{l-1}: If $\partial Z_j \subseteq D$, we have shown in Step 2
 that $Z_j$ is disjoint to $\Omega_\bullet$ and this implies $\partial
 \Omega_\bullet = \partial \Omega_\bullet \cap B = D$. In the second case, we
 infer from Step 3 in the above proof that $A \subseteq Z_j \subseteq
 \Omega_\bullet$ and this implies $\partial \Omega_\bullet = D \cup \partial B$.
 \end{proof}

Let us now conclude the proof of Theorem~\ref{t-hardy}. We first observe that in
both cases appearing in Corollary~\ref{c-002HE} the set $\partial
\Omega_\bullet$ is $m$-thick for some $m \in {]d-p,d-1]}$. In fact, $D$ is
$l$-thick for some $l \in {]d-p,d]}$ by assumption and using its local
representation as the graph of a Lipschitz function, it can easily be checked
that $\partial B$ is a $(d-1)$-set, hence $(d-1)$-thick owing to
Lemma~\ref{l-lsets with hausdorff content}. The claim follows from
Lemma~\ref{l-thickness for smaller values}. Altogether, Proposition~\ref{p-juha}
applies to our special choice of $\Omega_\bullet$. 

Now, let $\mathfrak{E}$ be the extension operator provided by
Assumption~\ref{t-hardy:iii} of Theorem~\ref{t-hardy}. In view of
Corollary~\ref{c-002HE} we can define an extension operator $\mathfrak E_\bullet
: W^{1,p}_D(\Omega) \to W^{1,p}_0(\Omega_\bullet)$ as follows: If $\partial
\Omega_\bullet =D$, then we put $\mathfrak E_\bullet v := \mathfrak E
v|_{\Omega_\bullet}$ and if $\partial \Omega_\bullet = D \cup \partial B$, then
we choose $\eta \in C_0^\infty(B)$ with the property $\eta \equiv 1$ on
$\overline \Omega$ and put $\mathfrak E_\bullet v := ( \eta \mathfrak E v
)|_{\Omega_\bullet}$. This allows us to apply Proposition~\ref{p-juha} to the
functions $\mathfrak E_\bullet u \in W^{1,p}_0(\Omega_\bullet)$, where $u$ is
taken from $W^{1,p}_D(\Omega)$. With a final help of Assumption~\ref{t-hardy:ii}
in Theorem~\ref{t-hardy} this gives
\begin{align*}
  \int_\Omega \left| \frac{u}{\d_D} \right|^p \; \dd \mathrm{x} &\le
	\int_\Omega \left| \frac{u}{\d_{\partial \Omega_\bullet}} \right|^p
	\; \dd \mathrm{x} \le \int_{\Omega_\bullet} \left|
	\frac{\mathfrak E_\bullet u}{\d_{\partial \Omega_\bullet}} \right|^p
	\; \dd \mathrm{x} \le  c \int_{\Omega_\bullet} |
	\nabla (\mathfrak E_\bullet u)|^p \; \dd \mathrm{x} \nonumber \\
 &\le c \|\mathfrak E_\bullet u\|^p_{W^{1,p}_0(\Omega_\bullet)}
  \le c \|u\|^p_{W^{1,p}_D(\Omega)} \le c \int_{\Omega} |\nabla u|^p
	\; \dd \mathrm{x}
\end{align*}
for all $u \in W^{1,p}_D(\Omega)$ and the proof is complete.

\begin{remark} \label{r-pointiii}
 \begin{enumerate} 
  \item At the first glance one might think that $\Omega_\bullet$ could always
	be taken as $B \setminus D$. The point is that this set need not be
	connected, as the following example shows. Take $\Omega = \{ \x
	: 1 < |\x| < 2 \}$ and $D = \{ \x : |\x| = 1 \}
	\cup \{ \x : |\x| = 2, \x_1 \ge 0 \}$. Obviously, if a
	ball $B$ contains $\overline \Omega$, then $B \setminus D$ cannot be
        connected. In the spirit of Lemma \ref{l-1}, the set $\Omega_\bullet$
        has here to be taken as $B \setminus (D \cup \{\x : |\x|
        < 1 \})$. Thus, the somewhat subtle, topological considerations above
	cannot be avoided in general.
  \item One might suggest that the procedure of this work is not limited to
	the proof of Hardy's inequality in the non-Dirichlet case. Possibly the
	 combination of an application of the extension operator $\mathfrak E
	/\mathfrak E_\bullet$ and the construction of $\Omega_\bullet$ may serve
	for the reduction of other problems on function spaces related to mixed
	boundary conditions to the pure Dirichlet case.
\end{enumerate}
\end{remark}

Finally, instead of its $l$-thickness we can also require that $D$ is an
$l$-set -- a condition that promises to be more common to
applications. One access to such a result is to
prove that the $l$-property of $\partial \Omega$ implies the $p$-fatness of
$\R^d \setminus \Omega$ -- a result which was first obtained by Maz'ya
\cite{mazcomm}. Knowing this, Hardy's inequality may then be deduced from the
results in \cite{lewis} or \cite{Wannebo}. Our approach is quite different and
simply rests on Proposition~\ref{p-juha} and Lemma~\ref{l-lsets with hausdorff
content}. So we can record the following.

\begin{corollary} \label{c-dminus1}
The assertion of Theorem \ref{t-hardy} remains valid if instead of
its $l$-thickness we require that $D$ is an $l$-set.
\end{corollary}

\section{The extension operator} \label{extens}
\noindent In this section we discuss the second condition in our main result
Theorem~\ref{t-hardy}, that is the extendability for $W^{1,p}_D(\Omega)$ within
the same class of Sobolev functions. We develop three abstract principles
concerning Sobolev extension.

\begin{itemize}
\item Dirichlet cracks can be removed: We open the possibility of passing from
$\Omega$ to another domain
$\Omega_\star$ with a reduced Dirichlet boundary part, while $\Gamma = \partial
\Omega \setminus D$ remains part of $\partial \Omega_\star$. In most cases this
improves the boundary geometry in the sense of Sobolev extendability, see the
example in the following Figure.

\begin{figure}[htbp]
\centerline{\includegraphics[scale=0.4]{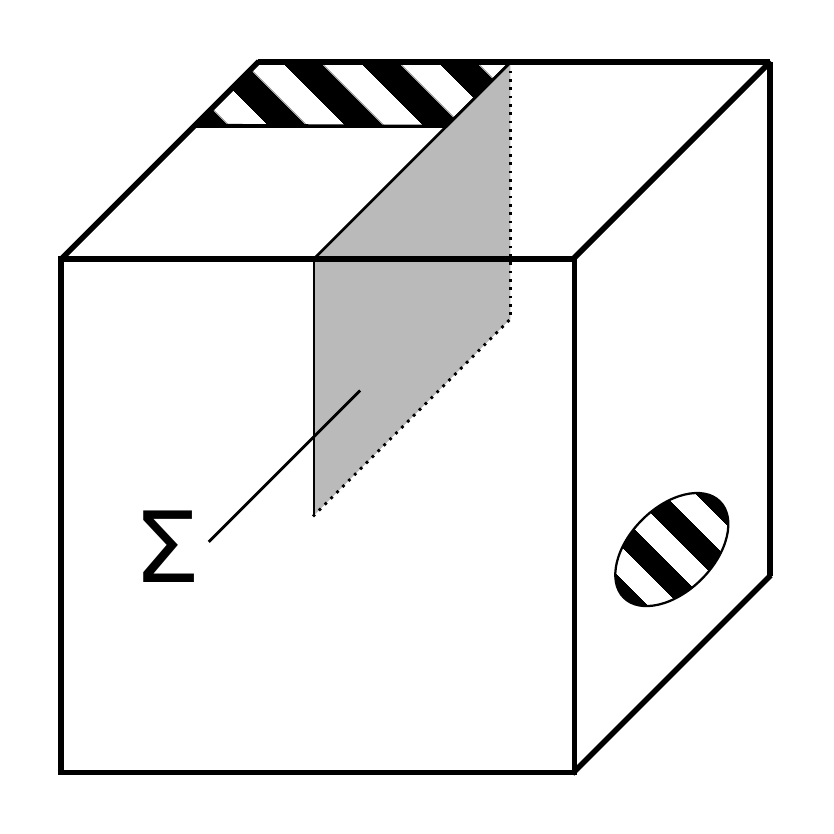}}
\caption{\label{fig-Kiste} 
The set $\Sigma$ does not belong to $\Omega$, and carries -- together with the
striped parts -- the Dirichlet condition.}
\end{figure}

\item Sobolev extendability is a local property: We show that only the local
geometry of the domain around the boundary
part $\Gamma$ plays a role for the existence of an extension operator.

\item Preservation of traces: We prove under very general geometric assumptions
that the extended
functions do have the adequate trace behavior on $D$ for \emph{every} extension
operator.
\end{itemize}
We believe that these results are of independent interest and therefore decided
to directly present them for higher-order Sobolev spaces $W_E^{k,p}$. In the
end we review some feasible commonly used geometric conditions which together
with our abstract principles really imply the corresponding extendability.

\subsection{Dirichlet cracks can be removed} \label{subsec: Removing cracks}
As in Figure~\ref{fig-Kiste} there may be boundary parts which carry a
Dirichlet condition and belong to the inner of the closure of the domain under
consideration. Then one can extend the functions on $\Lambda$ by $0$ to such a
boundary part, thereby enlarging the domain and simplifying the boundary
geometry. In the following we make this precise. 

\begin{lemma} \label{l-modigeb}
 Let $\Lambda \subset \R^d$ be a bounded domain and let $E \subset \partial
 \Lambda$ be closed. Define $\Lambda_\bigstar$ as the interior of the
 set $\Lambda \cup E$.
 Then the following hold true.
 \begin{enumerate} 
  \item The set $\Lambda_\bigstar$ is again a domain, $\Xi := \partial \Lambda
	\setminus E$ is a (relatively) open subset of $\partial
\Lambda_\bigstar$
	and $\partial \Lambda_\bigstar = \Xi \cup (E \cap \partial
	\Lambda_\bigstar)$.
 \item  Let $k \in \N$ and $p \in {[1,\infty[}$. Extending functions from
$W^{k,p}_E(\Lambda)$ by $0$ to $\Lambda_\bigstar$,
	one obtains an isometric extension operator
	$\mathrm{Ext}(\Lambda,\Lambda_\bigstar)$ from $W^{k,p}_E(\Lambda)$ onto
	$W^{k,p}_E(\Lambda_\bigstar)$.
 \end{enumerate}
\end{lemma} 

\begin{proof}
\begin{enumerate}
  \item Due to the connectedness of $\Lambda$ and the set inclusion
$\Lambda
\subset
	\Lambda_\bigstar \subset \overline \Lambda$, the set $\Lambda_\bigstar$
is
	also connected, and, hence a domain. Obviously, one has
	$\overline{\Lambda_\bigstar} = \overline \Lambda$. This, together with
the
	inclusion $\Lambda \subset \Lambda_\bigstar$ leads to
	$\partial \Lambda_\bigstar \subset \partial \Lambda$. Since $\Xi \cap
	\Lambda_\bigstar = \emptyset$, one gets $\Xi \subset
	\partial \Lambda_\bigstar$. Furthermore, $\Xi$ was relatively open in
	$\partial \Lambda$, so it is relatively open also in $\partial
	\Lambda_\bigstar$.

	The last asserted equality follows from $\partial \Lambda_\bigstar = (
	\Xi \cap \partial \Lambda_\bigstar) \cup (E \cap
	\partial \Lambda_\bigstar)$ and $\Xi \subset \partial \Lambda_\bigstar$.
  \item Consider any $\psi \in C^\infty_E(\R^d)$ and its restriction
	$\psi|_\Lambda$ to $\Lambda$. Since the support of $\psi$ has a positive
	distance to $E$, one may extend $\psi|_\Lambda$ by $0$ to the whole of
	$\Lambda_\bigstar$ without destroying the $C^\infty$-property.
        Thus, this extension operator
	provides a linear isometry from $C^\infty_E(\Lambda)$ \emph{onto}
	$C^\infty_E(\Lambda_\bigstar)$ (if both are equipped with the 
	$W^{k,p}$-norm). This extends to a linear extension operator
	$\mathrm{Ext}(\Lambda,\Lambda_\bigstar)$ from $W^{k,p}_E(\Lambda)$ onto
	$W^{k,p}_E(\Lambda_\bigstar)$, see the two following commutative
	diagrams:

\begin{minipage}{0.85\textwidth}
\begin{equation*}
\begin{tikzcd}[row sep=large, column sep=huge]
C^{\infty}_E (\mathbb{R}^d)
\arrow{d}[swap]{\operatorname*{restrict}_{\mathbb{R}^d\to\Lambda_\bigstar}}
\arrow{r}{\operatorname*{restrict}_{\mathbb{R}^d\to\Lambda}} 
& 
C^{\infty}_E(\Lambda) 
\arrow{ld}{\operatorname*{extend}_{\Lambda\to {\Lambda_\bigstar}}}
\\
C^{\infty}_E ({\Lambda_\bigstar})
\end{tikzcd}
\qquad
\begin{tikzcd}[row sep=large, column sep=huge]
	W^{k,p}_E (\mathbb{R}^d)
 	\arrow{r}{\operatorname*{restrict}_{\mathbb{R}^d\to\Lambda}}
\arrow{d}[swap]{\operatorname*{restrict}_{\mathbb{R}^d\to{\Lambda_\bigstar}}}
 	& W^{k,p}_E (\Lambda)
 	\arrow{ld}{\operatorname*{extend}_{\Lambda\to {\Lambda_\bigstar}}}
 	\\  
 	W^{k,p}_E ({\Lambda_\bigstar})
 \end{tikzcd}
\end{equation*}
\end{minipage}
\end{enumerate}
\end{proof}
 
\begin{remark} \label{r-arbitrar}
 \begin{enumerate}
  \item Note that no assumptions on $E$ beside closedness are necessary.
  \item Having extended the functions from $\Lambda$ to $\Lambda_\bigstar $, the
	'Dirichlet crack' $\Sigma$ in Figure~\ref{fig-Kiste} has vanished, and
	one ends up with the whole cube. Here the problem of extending Sobolev
	functions is almost trivial. We suppose that this is the generic case --
	at least for problems arising in applications.
 \end{enumerate} 
\end{remark}

The above considerations suggest the following procedure: extend the functions
from $W^{k,p}_E(\Lambda)$ first to $\Lambda_\bigstar$, and afterwards to the
whole of $\R^d$. The next lemma shows that this approach is universal.

\begin{lemma} \label{l-exten0}
 Let $k \in \N$ and $p \in {[1,\infty[}$. Let $\Lambda \subset \R^d$ be a
 bounded domain, let $E \subset \partial \Lambda$ be closed and as before define
 $\Lambda_\bigstar$ as the interior of the set $\Lambda \cup E$. Every linear,
 continuous extension operator $\mathfrak F : W^{k,p}_E(\Lambda)
 \to W^{k,p}_E(\R^d)$ factorizes as $\mathfrak F = {\mathfrak F}_\bigstar
 \mathrm{Ext}(\Lambda, \Lambda_\bigstar)$ through a linear, continuous extension
 operator ${\mathfrak F}_\bigstar : W^{k,p}_E(\Lambda_\bigstar) \to
 W^{k,p}_E(\R^d)$.
\end{lemma}

\begin{proof}
Let $\mathfrak S$ be the restriction operator from $W^{k,p}_E(\Lambda_\bigstar)$
to $W^{k,p}_E(\Lambda)$. Then we define, for every $f \in
W^{k,p}_E(\Lambda_\bigstar)$,
$\mathfrak F_\bigstar f:=\mathfrak F \mathfrak Sf$. We obtain
$\mathfrak F_\bigstar \mathrm{Ext}(\Lambda, \Lambda_\bigstar)=\mathfrak F
\mathfrak S
\mathrm{Ext}(\Lambda, \Lambda_\bigstar) =\mathfrak F$.
This shows that the factorization holds algebraically. However, one also has
 \begin{align*}
   \|{\mathfrak F}_\bigstar \mathrm{Ext}(\Lambda, \Lambda_\bigstar) f
	\|_{W^{k,p}_E(\R^d)} &= \| \mathfrak F f \|_{W^{k,p}_E(\R^d)} \le
	\| \mathfrak F \|_{\mathcal L(W^{k,p}_E(\Lambda);W^{k,p}_E(\R^d))}
	\|f\|_{W^{k,p}_E(\Lambda)} \\
   &= \|\mathfrak F \|_{\mathcal L(W^{k,p}_E(\Lambda);W^{k,p}_E(\R^d))}
	\| \mathrm{Ext}(\Lambda, \Lambda_\bigstar) f
	\|_{W^{k,p}_E(\Lambda_\bigstar)}.
  \qedhere
 \end{align*}
\end{proof}

Having extended the functions already to $\Lambda_\bigstar$ one may proceed
as follows: Since $E$ is closed, so is $E_\bigstar := E \cap \partial
\Lambda_\bigstar$. So, one can now consider the space
$W^{1,p}_{E_\bigstar}(\Lambda_\bigstar)$ and has the task to establish an
extension operator for this space -- while afterwards one has to take into
account that the original functions were $0$ also on the set $E \cap
\Lambda_\bigstar$ and have not been altered by the extension operator
thereon. However, note carefully that $E_\bigstar := E \cap \partial
\Lambda_\bigstar$ may have a worse geometry than $E$. For example, take Figure 2
and suppose that this time only $\Sigma$ forms the whole Dirichlet part of the
boundary. Then $E$ is a $(d-1)$-set whereas even $\H_{d-1}(E_\bigstar)=0$ holds.

To sum up, if one aims at an extension operator $\mathfrak E :
W^{k,p}_E(\Lambda) \to W^{k,p}_E(\R^d)$, one is free to modify the domain
$\Lambda$ to $\Lambda_\bigstar$. In most cases this improves the local geometry
concerning Sobolev extensions and we do not have examples where the situation
gets worse. Though we do not claim that this is, in a whatever precise sense,
the generic case.

\subsection{Sobolev extendability is a local property}\label{subsec:
extentability is local}
Below, we make precise in which sense Sobolev extendability is a local property.
We set up the following notation.

\begin{definition} \label{d-Sobolev extension domain}
A domain $\Lambda \subset \R^d$ is a \emph{$W^{k,p}$-extension domain} for given
$k \in \N$ and $p \in {[1,\infty[}$ if there exists a continuous extension
operator $\mathfrak{E}_{k,p}: W^{k,p}(\Lambda)\to W^{k,p}(\R^d)$. If $\Lambda$
is a $W^{k,p}$-extension domain for all $k \in \N$ and all $p \in {[1,\infty[}$
in virtue of the \emph{same} extension operator, then $\Lambda$ is
a \emph{universal Sobolev extension domain}.
\end{definition}

\begin{proposition} \label{p-extension}
Let $k \in \N$ and $p \in [1,\infty[$. Let $\Lambda$ be a bounded domain and
let $E$ be a closed part of its boundary. Assume that for every $\x \in
\overline{\partial \Lambda \setminus E}$ there is an open neighborhood $U_\x$ of
$ \x$ such that $\Lambda \cap U_\x$ is a $W^{k,p}$-extension domain. Then
there is a continuous extension operator
 \begin{align*}
  \mathfrak{E}_{k,p}: W^{k,p}_E(\Lambda) \to W^{k,p}(\R^d).
 \end{align*}
 Moreover, if each local extension operator $\mathfrak{E}_\x$ maps the space
 $W_{E_\x}^{k,p}(\Lambda \cap U_\x)$ into $W_{E_\x}^{k,p}(\R^d)$, where $E_\x :=
 \overline{E \cap U_\x} \subset \partial(\Lambda \cap U_\x)$, then also
 \begin{align*}
  \mathfrak{E}_{k,p}: W^{k,p}_E(\Lambda) \to W^{k,p}_E(\R^d).
 \end{align*}
\end{proposition}

\begin{proof}
 For the construction of the extension operator let for every $\x \in
 \overline{\partial \Lambda \setminus E}$ denote
 $U_\x$ the open neighborhood of $\x$ from the assumption.
 Let $U_{\x_1}, \ldots, U_{\x_n}$ be a finite subcovering of 
 $\overline{\partial \Lambda \setminus E}$. Since the compact set
 $\overline{\partial \Lambda \setminus E}$ is contained in the open set
 $\bigcup_j U_{\x_j}$, there is an $\epsilon > 0$, such that the sets
 $U_{\x_1},\ldots,U_{\x_n}$, together with the open set $U :=
 \{\mathrm y \in \R^d : \mathrm{dist}(\mathrm y, \overline{\partial \Lambda
 \setminus E}) > \epsilon \}$, form an open covering of $\overline \Lambda$.
 Hence, on $\overline \Lambda$ there is a $C_0^\infty$-partition of unity $\eta,
 \eta_1, \ldots, \eta_n$, with the properties $\mathrm{supp}(\eta) \subset U$,
 $\mathrm{supp}(\eta_j) \subset U_{\x_j}$.

 Assume $\psi \in C^\infty_E(\Lambda)$. Then $\eta \psi \in
 C^\infty_0(\Lambda)$. If one extends this function by $0$ outside of $\Lambda$,
 then one obtains a function $\varphi \in
 C^\infty_{\partial \Lambda}(\R^d) \subset C^\infty_E(\R^d) \subset
 W^{k,p}_E(\R^d)$ with the property $\|\varphi\|_{W^{k,p}(\R^d)} =
 \|\eta \psi\|_{W^{k,p}(\Lambda)}$.

 Now, for every fixed $j \in \{1, \ldots, n\}$, consider the function $\psi_j
 := \eta_j \psi \in W^{k,p}(\Lambda \cap U_{\x_j})$. Since $\Lambda \cap
 U_{\x_j}$ is a $W^{k,p}$-extension domain by assumption, there is an
 extension of $\psi_j$ to a $W^{k,p}(\R^d)$-function $\varphi_j$ together with
 an estimate $\|\varphi_j\|_{W^{k,p}(\R^d)} \le c
 \|\psi_j\|_{W^{k,p}( \Lambda \cap U_{\x_j})}$, where $c$ is independent
 from $\psi$. Clearly, one has a priori no control on the behavior of
 $\varphi_j$ on the set $\Lambda \setminus U_{\x_j}$. In particular
 $\varphi_j$ may there be nonzero and, hence, cannot be expected to coincide
 with $\eta_j \psi$ on the whole of $\Lambda$. In order to correct
 this, let $\zeta_j$ be a $C^\infty_0(\R^d)$-function which is identically $1$
 on $\mathrm{supp}(\eta_j)$ and has its support in $U_{\x_j}$. Then
 $\eta_j \psi$ equals $\zeta_j \varphi_j $ on all of $\Lambda$. Consequently,
 $\zeta_j \varphi_j$ really is an extension of $\eta_j \psi$ to the whole of
 $\R^d$ which, additionally, satisfies the estimate
 \[ \| \zeta_j \varphi_j \|_{W^{k,p}(\R^d)} \le c \| \varphi_j
	\|_{W^{k,p}(\R^d)}\le c \| \eta_j \psi
	\|_{W^{k,p}(\Lambda \cap U_{\x_j})} \le c
	\|\psi\|_{W^{k,p}(\Lambda)},
 \]
 where $c$ is independent from $\psi$. Thus, defining $\mathfrak E_{k,p}(\psi) =
 \varphi + \sum_j \zeta_j \varphi_j$ one gets a linear, continuous extension
 operator from $C^\infty_E(\Lambda)$ into $W^{k,p}(\R^d)$. By density,
 $\mathfrak{E}_{k,p}$ uniquely extends to a linear, continuous operator
 \begin{equation*}
  \mathfrak{E}_{k,p}:W^{k,p}_E(\Lambda) \to W^{k,p}(\R^d).
\end{equation*} 

 Finally, assume that the local extension operators map $W^{k,p}_{E_{\x_j}}
(\Lambda \cap U_{\x_j})$ into $W^{k,p}_{E_{\x_j}}(\R^d)$. Using the notation
above, this means that $\varphi_j$ can be approximated in $W^{k,p}(\R^d)$ by a
sequence from $C_{E_{\x_j}}^\infty(\R^d)$. Since $\zeta_j$ is supported in
$U_{x_j}$, multiplication by $\zeta_j \in C_0^\infty(\R^d)$ maps
$C_{E_{\x_j}}^\infty(\R^d)$ into $C_{E}^\infty(\R^d)$ boundedly with respect
to the $W^{k,p}(\R^d)$-topology. Hence, $\zeta_j \varphi_j \in
W^{k,p}_E(\R^d)$. Since in any case $\varphi \in W^{k,p}_E(\R^d)$, the
conclusion follows.
\end{proof}

\begin{remark} \label{r-unif}
By construction one gets uniformity for $\mathfrak{E}$ with respect to $p$ and
$k$ if one invests the respective uniformity concerning the extension property
for the local domains $\Lambda \cap U_\x$. In particular, one obtains an
extension operator that is bounded from $W_E^{k,p}(\Lambda)$ into
$W^{k,p}(\R^d)$ for each $k \in \N$ and each $p \in {[1,\infty[}$ if the local
domains are universal Sobolev extension domains.
\end{remark}

\subsection{Preservation of traces} \label{Subsec: preservation of traces}
Proposition~\ref{p-extension} allows to construct Sobolev extension operators
from $W^{k,p}_D(\Omega)$ into $W^{k, p}(\R^d)$ and gives a sufficient condition
for preservation of the Dirichlet condition. In this section we prove that in
fact \emph{every} such extension operator has this feature. Recall that this is
the crux of the matter in Assumption~\ref{t-hardy:iii} of
Theorem~\ref{t-hardy}.
The key lemma is the following.

\begin{lemma} \label{l-trace}
Let $k \in \N$ and $p \in {]1,\infty[}$. Let $\Lambda \subset \R^d$ be a
domain, let $E \subset \partial \Lambda$ be closed and let $\mathfrak{E}_{k,p}:
W^{k,p}_E(\Lambda) \to W^{k,p}(\R^d)$ be a bounded extension operator. Any of
the following conditions guarantees that $\mathfrak{E}_{k,p}$ in fact maps into
$W_E^{k,p}(\R^d)$.
\begin{enumerate}
 \item For $(k,p)$-quasievery $\y \in E$ balls around $\mathrm y$ in
 $\Lambda$ have \emph{asymptotically nonvanishing relative volume}, i.e.
 \begin{align} \label{e-asymp}
  \liminf_{r \to 0} \frac {|B(\y, r) \cap \Lambda)|}{r^d} > 0.
 \end{align}

 \item The set $E$ is an $l$-set for some $l \in {]d-p,d]}$ and \eqref{e-asymp}
 holds for $\H_l$-almost every $\y \in E$.

 \item There exists $q > d$ such that $\mathfrak{E}_{k,p}$ maps
 $C_E^\infty(\Lambda)$ into $W^{k,q}(\R^d)$.
\end{enumerate}
\end{lemma}

\begin{proof}
As $C_E^\infty(\Omega)$ is dense in $W_E^{k,p}(\Lambda)$ and since
$\mathfrak{E}_{k,p}$ is bounded, it suffices to prove that given $v \in
C_E^\infty$ the function $u:= \mathfrak{E}_{k,p} v$ belongs
to $W_E^{k,p}(\R^d)$. The proof of (i) is inspired by \cite[pp.\
190-192]{ziemer}. Easy modifications of the argument will yield (ii) and
(iii). 

\begin{enumerate}

\item Fix an arbitrary multiindex $\beta$ with $|\beta|\leq k-1$. Let
  $\mathfrak{D}^\beta \u$ be the representative of the distributional derivative
  $D^\beta u$ of $u$ defined $(k-|\beta|,p)$-q.e.\ on $\R^d$ via
  \begin{align*}
  \mathfrak{D}^\beta \u(\y) := \lim_{r \to 0}
  \frac{1}{|B(\y,r)|}\int_{B(\y,r)}D^\beta u(\x) \; \d \x.
  \end{align*}
  Recall from \eqref{e-Lebesgue} that then
  \begin{align}
  \label{e-Lebesgue for rep}
  \begin{split}
  &\lim_{r \to 0} \frac {1}{|B(\y,r)|} \int_{B(\y,r)} |\mathfrak{D}^\beta \u(
  \x)-\mathfrak{D}^\beta \u(\y)| \; \d \x \\
  &\leq
  \lim_{r \to 0} \bigg(\frac {1}{|B(\y,r)|} \int_{B(\y,r)} |\mathfrak{D}^\beta
  \u( \x)-\mathfrak{D}^\beta \u(\y)|^p \; \d \x \bigg)^{1/p} = 0.
  \end{split}
  \end{align}
  holds for $(k-|\beta|,p)$-q.e.\ $\y \in \R^d$. Since \eqref{e-asymp} holds for
  $(k,p)$-quasievery $\y \in E$, it a fortiori holds for
  $(k-|\beta|,p)$-quasievery such $\y$. Let now $N \subset \R^d$ be the
  exceptional set such that on $\R^d \setminus N$ the function
  $\mathfrak{D}^\beta \u$ is defined and satisfies \eqref{e-Lebesgue for rep}
  and such that \eqref{e-asymp} holds for every $\y \in E \setminus N$. Owing to
  Theorem~\ref{t-coincid} the claim follows once we have shown
  $\mathfrak{D}^\beta  \u(\y) = 0$ for all $\y \in E \setminus N$. 

  For the rest of the proof we fix $y \in E \setminus N$. For $r>0$ we
  abbreviate $B(r):=B(\y,r)$ and define
  \begin{align}\label{d-set_quasicontinuity1}
  W_j:=\{ \x \in \R^d \setminus N: |\mathfrak{D}^\beta \u(\x) - D^\beta \u(\y)|
  > 1/j\}.
  \end{align}
  Thanks to \eqref{e-Lebesgue for rep} for each $j \in \N$ we can choose some 
  $r_j > 0$ such that $|B(r) \cap W_j| < 2^{-j} |B(r)|$ holds for all $r \in
  {]0,r_j]}$. Clearly, we can arrange that the sequence $\{r_j\}_j$
  is decreasing. Now, 
  \begin{align} \label{d-set_quasicontinuity2}
  W:= \bigcup_{j \in \N} \Big \{ \big(B(r_j) \setminus B(r_{j+1})\big) \cap W_j
  \Big \}
  \end{align}
  has vanishing Lebesgue density at $\y$, i.e.\ $r^{-d}|B(r) \cap W|$ vanishes
  as $r$ tends to $0$: Indeed, if $r \in {]r_{l+1}, r_l]}$, then
  \begin{align*}
  |B(r) \cap W| 
  &\leq \Big| \big(B(r) \cap W_l\big) \cup \bigcup_{j \geq l+1} 
  \big(B(r_j) \cap W_j \big)\Big| \\
  &\leq 2^{-l}|B(r)| + \sum_{j \geq l+1} 2^{-j}|B(r_j)| 
  \leq 2^{-l+1} |B(r)|.
  \end{align*}
  Now, \eqref{e-asymp} allows to conclude
  \begin{align*}
  \liminf_{r \to 0} \frac {|B(r) \cap 
  \Lambda \cap (\R^d \setminus W))|}{r^d} > 0.
  \end{align*}
  Since $u$ is an extension of $v \in C^\infty_E(\Lambda)$ and $\y$ is an
  element of $E$ it holds $\mathfrak{D}^\beta \u = 0$ a.e.\ on $B(r) \cap
  \Lambda$ with respect to the $d$-dimensional Lebesgue measure if $r>0$ is
  small enough. The previous inequality gives $|B(r) \cap \Lambda \cap (\R^d
  \setminus W))| > 0$ if $r>0$ is small enough. In particular, there exists a
  sequence $\{\x_j\}_j$ in $\R^d \setminus W$ approximating $\y$ such that
  $\mathfrak{D}^\beta \u(\x_j) =  0$ for all $j \in \N$. Now, the upshot is that
  the restriction of  $\mathfrak{D}^\beta \u$ to $\R^d \setminus W$ is
  continuous at $\y$ since if $\x  \in \R^d \setminus W$ satisfies $|\x - \y|
  \leq r_j$ then by construction  $|\mathfrak{D}^\beta \u(\x) -
  \mathfrak{D}^\beta \u(\y)| \leq 1/j$. Hence,  $\mathfrak{D}^\beta \u(\y) = 0$
  and the proof is complete.

\item If $E$ is an $l$-set for some $l \in {]d-p,d]}$, then we can appeal to
  Theorem~\ref{t-coincid-mitrea} rather than Theorem~\ref{t-coincid} and the
  same argument as in (i) applies.

\item By assumption $u \in W_E^{k,q}(\R^d)$, where $q > d$. By Sobolev
  embeddings each distributional derivative $D^\beta u$, $|\beta| \leq
  k-1$, has a continuous representative $\mathfrak{D}^\alpha \u$. As each $\y
  \in E \subset \partial \Lambda$ is an accumulation point of $\Lambda \setminus
  E$ and since $D^\alpha u = D^\alpha v$ holds almost everywhere on $\Lambda$,
  the representative $\mathfrak{D}^\alpha \u$ must vanish everywhere on $E$ and
  Theorem~\ref{t-coincid} yields $u \in W_E^{k,p}(\R^d)$ as required.
  \qedhere
\end{enumerate}
\end{proof}

\begin{remark} \label{r-trace}
If $\Lambda$ is a $d$-set and $E$ a $(d-1)$-set, then Lemma~\ref{l-trace} is
proved in \cite[Sec.~VIII.1]{jons}.
\end{remark}

We can now state and prove the remarkable result that every Sobolev extension
operator that is constructed by localization techniques as
in Proposition~\ref{p-extension} preserves the Dirichlet condition.

\begin{theorem}\label{t-extension always preserves trace}
Let $k \in \N$ and $p \in [1,\infty[$. Let $\Lambda$ be a bounded domain and
let $E$ be a closed part of its boundary. Assume that for every $\x \in
\overline{\partial \Lambda \setminus E}$ there is an open neighborhood $U_\x$ of
$ \x$ such that $\Lambda \cap U_\x$ is a $W^{k,p}$-extension domain. Then
there exists a continuous extension operator
 \begin{align*}
  \mathfrak{E}_{k,p}: W^{k,p}_E(\Lambda) \to W_E^{k,p}(\R^d).
 \end{align*}
\end{theorem}

For the proof we recall the following result of Hai\l{}asz, Koskela
and Tuominen.

\begin{proposition}[{\cite[Thm.\ 2]{hajla}}]\label{p-extension implies dset}
If a domain $\Lambda \subset \R^d$ is a $W^{k,p}$-extension domain for some $k
\in \N$ and $p \in {[1,\infty[}$, then it is a $d$-set.
\end{proposition}

\begin{proof}[Proof of Theorem~\ref{t-extension always preserves trace}]
According to Proposition~\ref{p-extension} it suffices to check that each local
extension operator $\mathfrak{E}_\x$ maps $W_{E_\x}^{k,p}(\Lambda \cap U_\x)$
into $W_{E_\x}^{k,p}(\R^d)$, where $E_\x := \overline{E \cap U_\x}
\subset \partial(\Lambda \cap U_\x)$. Owing to Proposition~\ref{p-extension
implies dset} the $W^{k,p}$-extension domain $\Lambda \cap U_\x$
is a $d$-set and as such satisfies \eqref{e-asymp} around every of its boundary
points. So, Lemma~\ref{l-trace}.(i) yields the claim.
\end{proof}

\begin{remark}\label{r-extension always preserves trace}
The extension operator in Theorem~\ref{t-extension
always preserves trace} is the same as in Proposition~\ref{p-extension}. Hence,
the former result asserts that every Sobolev extension operator built by
the common gluing-together of local extension operators automatically preserves
the Dirichlet condition on $E$ under the mere assumption that this set is
closed. Moreover, all uniformity properties as in Remark~\ref{r-unif} remain
valid. 
\end{remark}

\subsection{Geometric conditions} \label{Subsec: geometric conditions}
In this subsection we finally review common geometric conditions on the boundary
part $\overline {\partial \Lambda \setminus E}$ such that the local sets
$\Lambda \cap U_{\x}$ really admit the Sobolev extension property required in
Proposition~\ref{p-extension}. 

A first condition, completely sufficient for the treatment of most real world
problems, is the following Lipschitz condition.

\begin{definition}\label{d-Lipschitzdomain}
A bounded domain $\Lambda \subset \R^d$ is called \emph{bounded Lipschitz
domain} if for each $\x \in \partial \Lambda$ there is an open neighborhood
$U_\x$ of $\x$ and a bi-Lipschitz mapping $\phi_\x$ from $U_\x$ onto a cube,
such that $\phi_\x(\Lambda \cap U_\x)$ is the (lower) half cube and $\partial
\Lambda \cap U_\x$ is mapped onto the top surface of this half cube.
\end{definition}

It can be proved by elementary means that bounded Lipschitz domains are
$W^{1,p}$-extension domains for every $p \in [1,\infty[$, cf.\ e.g.\
\cite{Giusti} for the case $p=2$. In fact, already the following
$(\eps,\delta)$-condition of Jones \cite{jones} assures the existence of a
universal Sobolev extension operator. 

\begin{definition} \label{d-Jones}
 Let $\Lambda \subset \R^d$ be a domain and $\varepsilon, \delta >0$. Assume
 that any two points $\x, \mathrm y \in \Lambda$, with distance not
 larger than $\delta$, can be connected within $\Lambda$ by a rectifiable arc
 $\gamma$ with length $l(\gamma)$, such that the following two conditions are
 satisfied for all points $\mathrm z$ from the curve $\gamma$:
 \begin{align*}
 l(\gamma) \le \frac{1}{\varepsilon } \|\x - \mathrm y\|, \quad
	\text{and} \quad
	\frac{\|\x - \mathrm z\| \|\mathrm y -\mathrm z\|}
		{\|\x -\mathrm y\|}
	\le \frac{1}{\varepsilon} \mathrm{dist}(\mathrm z, \Lambda^c).
 \end{align*}
 Then $\Lambda$ is called \emph{$(\varepsilon, \delta)$-domain}.
\end{definition}

\begin{theorem}[Rogers] \label{t-Rogers}
 Each $(\varepsilon,\delta)$-domain is a universal Sobolev extension domain.
\end{theorem}

\begin{remark} \label{r-uniform}
\begin{enumerate}
 \item Theorem~\ref{t-Rogers} is due to Rogers \cite{rogers} and generalizes the
celebrated result of Jones \cite{jones}. \emph{Bounded}
$(\varepsilon,\delta)$-domains are known  to be \emph{uniform domains}, see
\cite[Ch.~4.2]{Vai} and also \cite{jones, martio, martiosarv, martin} for
further information. In particular, every bounded Lipschitz
domain is an $(\eps,\delta)$-domain, see e.g.\
\cite[Rem.~5.11]{Darmstadt_KatoMixed} for a sketch of proof.
 \item Although the uniformity property is not necessary for a domain to be a
Sobolev extension domain \cite{Yang} it seems presently to be the broadest known
class of domains for which this extension property holds -- at least if one aims
at all $p \in {]1,\infty[}$. For example Koch's snowflake is an
$(\varepsilon,\delta)$-domain \cite{jones}.
\end{enumerate}
\end{remark}

Plugging in Rogers extension operator into Theorem~\ref{t-extension always
preserves trace} lets us re-discover \cite[Thm.~1.3]{mitrea}
in case of bounded domains and $p$ strictly between $1$ and $\infty$. We even
obtain a universal extension operator that simultaneously acts on all
$W_E^{k,p}$-spaces and at the same time our argument reveals that the
preservation of the trace is irrespective of the specific structure of
Jones' or Roger's extension operators. 

We believe that this sheds some more light also on \cite[Thm.~1.3]{mitrea}
though -- of course -- our argument cannot disclose the fundamental assertions
on the support of the extended functions obtained in \cite{mitrea} by a careful
analysis of Jones' extension operator. We summarize our observations in the
following theorem.

\begin{theorem} \label{t-rediscovering mitrea}
 Let $\Lambda$ be a bounded domain and let $E$ be a closed part of its boundary.
 Assume that for every $\x \in \overline{\partial \Lambda \setminus E}$ there
 is an open neighborhood $U_\x$ of $ \x$ such that $\Lambda \cap U_\mathrm
 x$ is a bounded Lipschitz or, more generally, an $(\eps, \delta)$-domain for
 some values $\eps, \delta > 0$. Then there exists a universal operator
 $\mathfrak{E}$ that restricts to a bounded extension operator
 $W_E^{k,p}(\Lambda) \to W_E^{k,p}(\R^d)$ for each $k \in \N$ and each $p \in
 {]1,\infty[}$.
\end{theorem}

\section{Poincar\'e's inequality}
\label{sec-poincare}
\noindent In this section we will discuss sufficient conditions for Poincar\'e's
inequality, thereby unwinding Assumption~\ref{t-hardy:ii} of
Theorem~\ref{t-hardy}. Our aim is not greatest generality as e.g.\
in \cite{mazsob} for functions defined on the whole of $\R^d$, but to include
the aspect that our functions are only defined on a domain. Secondly, our
interest is to give very general, but in some sense \emph{geometric} conditions,
which may be checked more or less `by appearance' -- at least for problems
arising from applications. 

The next proposition gives a condition that assures that a closed
subspace of $W^{1,p}$ may be equivalently normed by the $L^p$-norm of the
gradient of the corresponding functions only. We believe
that this might also be of independent interest, compare also
\cite[Ch.~4]{ziemer}. Throughout $\mathds 1$ denotes the function that is
identically one.

\begin{proposition} \label{t-constant}
 Let $\Lambda \subset \R^d$ be a bounded domain and suppose $p \in
 {]1,\infty[}$. Assume that $X$ is a closed subspace of $W^{1,p}(\Lambda)$ that
 does not contain $\mathds 1$ and for which the restriction of the canonical
 embedding $W^{1,p}(\Lambda) \hookrightarrow L^p(\Lambda)$ to $X$ is
 compact. Then $X$ may be equivalently normed by $v \mapsto \bigl( \int_\Lambda
 |\nabla v|^p \mathrm{d} \x \bigr)^{1/p}$.
\end{proposition} 

\begin{proof}
 First recall that both $X$ and $L^p(\Lambda)$ are reflexive. In order to
 prove the proposition, assume to the contrary that there exists a sequence
 $\{v_k\}_k$ from $X$ such that
 \begin{align*}
  \frac{1}{k} \|v_k\|_{L^p(\Lambda)} \ge \|\nabla v_k\|_{L^p(\Lambda)}.
 \end{align*}
 After normalization we may assume $\|v_k\|_{L^p(\Lambda)} = 1$ for every $k
 \in \N$. Hence, $\{\nabla v_k\}_k$ converges to $0$ strongly in $L^p(\Lambda)$.
 On the other hand, $\{v_k\}_k$ is a bounded sequence in $X$ and hence contains
 a subsequence $\{v_{k_l}\}_l$ that converges weakly in $X$ to an element $v \in
 X$. Since the gradient operator $\nabla : X \to L^p(\Lambda)$ is
 continuous, $\{\nabla v_{k_l}\}_l$ converges to $\nabla v$ weakly in 
 $L^p(\Lambda)$. As the same sequence converges to $0$ strongly in
 $L^p(\Lambda)$, the function $\nabla v$ must be zero and hence $v$ is constant.
 But by assumption $X$ does not contain constant functions except for $v = 0$.
 So, $\{v_{k_l}\}_l$  tends to $0$ weakly in $X$. Owing to the compactness of
 the embedding $X \hookrightarrow L^p(\Lambda)$, a subsequence of
 $\{v_{k_l}\}_l$ tends to $0$ strongly in $L^p(\Lambda)$. This contradicts
 the normalization condition $\|v_{k_l}\|_{L^p(\Lambda)} = 1$.
\end{proof}

\begin{remark} \label{r-bederfuellt}
 It is clear that in case $X = W^{1,p}_D(\Omega)$ the embedding $X
 \hookrightarrow L^p(\Omega)$ is compact, if there exists a continuous
 extension operator $\mathfrak E : W^{1,p}_D(\Omega) \to W^{1,p}(\R^d)$. Hence,
 the compactness of this embedding is no additional requirement in view of
 Theorem~\ref{t-hardy}.
\end{remark}

In the case that $E$ is $l$-thick, the following lemma presents two
conditions that are particularly easy to check and entail $\mathds 1 \notin
W^{1,p}_E(\Lambda)$. Loosely speaking, some knowledge on the common frontier of
$E$ and $\partial \Lambda \setminus E$ is required: Either not every point of
$E$ should lie thereon or $\partial \Lambda$ must not be too wild around this
frontier. 

\begin{lemma} \label{l-constantzero}
Let $p \in {]1,\infty[}$, let $\Lambda$ be a bounded domain and let $E \subset
\partial \Lambda$ be $l$-thick for some $l \in {]d-p,d]}$. Both of the
following conditions assure $\mathds 1 \notin W_E^{1,p}(\Lambda)$.
 \begin{enumerate}
  \item \label{l-constantzero:i} The set $E$ admits at least one relatively
  inner point $\x$. Here, `relatively inner' is with  respect to $\partial
  \Lambda$ as ambient topological space.
  \item \label{l-constantzero:ii} For every $\x \in \overline{\partial \Lambda
\setminus E}$ there is an open neighborhood $U_\x$ of $ \x$ such that $\Lambda
\cap U_\mathrm x$ is a $W^{1,p}$-extension domain.
\end{enumerate}
\end{lemma}

\begin{proof}
We treat both cases separately.
 \begin{enumerate}
  \item Assume the assertion was false and $\mathds{1} \in
	W_E^{1,p}(\Lambda)$. Let $\x$ be the inner point of $E$ from the
	hypotheses and let $B := B(\x,r)$ be a ball that does not
	intersect $\partial \Lambda \setminus E$. Put $\frac{1}{2}B:= B(\x
	,\frac{r}{2})$ and let $\eta \in C^\infty_0(B)$ be such that $\eta
	\equiv 1$ on $\frac{1}{2} B$. We distinguish whether or not $\x$ is an
	interior point of $\overline{\Lambda}$.

	First, assume it is not. For every $\psi \in C^\infty_E(\Lambda)$
	the function $\eta \psi$ belongs to $W^{1,p}_0(\Lambda \cap B)$ and as
	such admits a $W^{1,p}$-extension $\widehat {\eta \psi}$ by zero to the
	whole of $\R^d$. In particular,
	\begin{align*}
	  \widehat{\eta \psi} (\y) =
		\begin{cases}
		  \psi (\y), & \text{if} \quad \y \in 
			  \frac{1}{2} B \cap \Lambda\\
		  0, & \text{if} \quad \y \in \frac{1}{2} B
			  \setminus \Lambda
		\end{cases}
	\end{align*}
	and consequently,
	\begin{align*}
	  \|\nabla \widehat{\eta \psi}\|_{L^p(\frac{1}{2}B)} =
		\|\nabla \psi\|_{L^p(\frac{1}{2} B \cap \Lambda)}.
	\end{align*}
	Since by assumption $\mathds{1}$ is in the
	$W^{1,p}(\Lambda)$-closure of $C^\infty_E(\Lambda)$ and since the
        mappings $W^{1,p}_E(\Lambda) \ni \psi \mapsto \nabla \widehat{\eta \psi}
	\in L^p(\frac{1}{2}B)$ and $W^{1,p}_E(\Lambda) \ni \psi
	\mapsto \nabla \psi \in L^p(\Lambda \cap \frac{1}{2}B)$
	are continuous, the previous equation extends to $\psi = \mathds{1}$:
	\begin{align*}
	  \|\nabla \widehat{\eta \mathds{1}}\|_{L^p(\frac{1}{2}B)}
	  = \|\nabla \mathds{1}\|_{L^p(\frac{1}{2} B \cap \Lambda)}
	  = 0.
	\end{align*}
	On the other hand $\x$ is not an inner point of
	$\overline{\Lambda}$ so that in particular $\frac{1}{2} B \setminus
	\overline{\Lambda}$ is non-empty. Since this set is open, $|\frac{1}{2}
	B \setminus \overline{\Lambda}| > 0$.
	Recall that by construction $\widehat{\eta \mathds{1}} \in W^{1,p}(B)$
	vanishes a.e.\ on $\frac{1}{2} B \setminus \overline{\Lambda}$. Hence,
	for some $c > 0$ the Poincar\'e inequality
	\begin{align*}
	  \|\widehat{\eta \mathds{1}}\|_{L^p(\frac{1}{2} B)} \le c
		\|\nabla \widehat{\eta \mathds{1}}
		\|_{L^p(\frac{1}{2} B)},
	\end{align*}
	holds, cf.\@ \cite[Thm.~4.4.2]{ziemer}. However, we already know that
	the right hand side is zero, whereas the left hand side equals
	$|\frac{1}{2} B \cap \Lambda|^{1/p}$, which is nonzero
	since $\frac{1}{2} B \cap \Lambda$ is nonempty and open --  a
	contradiction. 

	Now, assume $\x$ is contained in the interior of $\overline{\Lambda}$.
	Upon diminishing $B$ we may assume $B \subset \overline{\Lambda}$. For
	every $\psi \in C_E^\infty(\R^d)$ we have $\eta \psi \in
	C_E^\infty(\R^d)$ with an estimate
	\begin{align*}
	\|\eta \psi\|_{W^{1,p}(\R^d)} \leq c \|\psi\|_{W^{1,p}(B)} 
	= c \Big(\int_B |\psi|^p + |\nabla \psi|^p \; \d
	  \x\Big)^{1/p}
	\end{align*}
	for some constant $c>0$ depending only on $\eta$ and $p$. By our choice
	of $B$ split
	\begin{align*}
	 B = B \cap \overline{\Lambda} 
	    = (B \cap \Lambda)\cup (B \cap \partial \Lambda) 
	    = (B \cap \Lambda)\cup (B \cap E).
	\end{align*}
	Since $\psi$ vanishes in a neighborhood of $E$,
	\begin{align}
	\label{Eq1: const}
	\|\eta \psi\|_{W^{1,p}(\R^d)} 
	\leq c \Big(\int_{B \cap \Lambda} |\psi|^p + |\nabla \psi|^p \; \d
	  \x\Big)^{1/p}
	\leq c\|\psi\|_{W^{1,p}(\Lambda)}.
	\end{align}
	Taking into account $\eta \equiv 1$ on $\frac{1}{2}B$, the same
	reasoning gives
	\begin{align}
	\label{Eq2: const}
	\int_{\frac{1}{2} B} |\nabla (\eta \psi)|^p \; \d \x
	= \int_{\frac{1}{2} B} |\nabla \psi|^p \; \d \x
	\leq \int_\Lambda |\nabla \psi|^p \; \d \x.
	\end{align}
	By assumption there is a sequence $\{\psi_j\}_j \subset
	C_E^\infty(\Lambda)$ tending to $\mathds{1}$
	in the $W^{1,p}(\Lambda)$-topology. Due to \eqref{Eq1: const} and the
	choice of $\eta$, the sequence $\{\eta \psi_j\}_j
	\subset C_E^\infty(\R^d)$ then tends to some $u \in W_E^{1,p}(\R^d)$
	satisfying $u = 1$ a.e.\ on $\frac{1}{2}B \cap \Lambda$. Due to
	\eqref{Eq2: const}, $\nabla u = 0$ a.e.\ on $\frac{1}{2}B$, meaning that
	$u$ is constant on this set. Since $\frac{1}{2}B \cap \Lambda$ as a
	non-empty open set has positive Lebesgue measure, all this can only
	happen if $u = 1$ a.e.\ on $\frac{1}{2}B$. Hence,
	\begin{align*}
	 \lim_{r \to 0} \frac{1}{|B(\y,r)|} \int_{B(\y,r)} u \; \d \x = 1
	\end{align*}
	for every $\y \in \frac{1}{3} B \cap E$, which by
	Theorem~\ref{t-coincid} is only possible if $C_{1,p}(\frac{1}{3} B \cap
	E) = 0$. By Theorem~\ref{t-coincid} this in turn implies
	$\H_l^\infty(\frac{1}{3} B \cap E) = 0$	in contradiction to the
	$l$-thickness of $E$.

  \item Again assume the assertion was false. Then by \ref{l-constantzero:i}
	there exists some $\x \in E$ that is not an inner point of $E$ with
	respect to $\partial \Lambda$. Hence $\x$ is an accumulation point of
	$\partial \Lambda \setminus E$ and by assumption there is a neighborhood
	$U = U_{\x}$ of $\x$ such that $\Lambda \cap U$ is a $W^{1,p}$ extension
	domain. We denote the corresponding extension operator by
	$\mathfrak{E}$.
	We shall localize the assumption $\mathds{1} \in W_E^{1,p}(\Lambda)$
	within $U$ to arrive at a contradiction.

	To this end, let $r_0 > 0$ be such that $\overline{B(\x,r_0)} \subset
	U$ and let $\eta \in C_0^\infty(U)$ be such that $\eta \equiv 1$ on
	$B(\x,r_0)$. Then also $\eta = \eta \mathds{1}	\in
	W_E^{1,p}(\Lambda)$ and in particular $\eta|_{\Lambda \cap U}$
	belongs to $W_F^{1,p}(\Lambda \cap U)$, where $F :=
	\overline{B(\x,r_0/2) \cap E} \subset \partial(\Lambda \cap U)$. Recall
	from Proposition~\ref{p-extension implies dset} that the
	$W^{1,p}$-extension domain $\Lambda \cap U$ satisfies in particular
	\begin{align*}
	  \liminf_{r \to 0} \frac {|B(\mathrm y,r) \cap \Lambda \cap U)|}{r^d} >
	  0.
	\end{align*} 
	around every $\y \in \partial(\Lambda \cap U)$. Thus,
	Lemma~\ref{l-trace}(i) yields
	$u:=\mathfrak{E}(\eta|_{\Lambda \cap U}) \in W_F^{1,p}(\R^d)$.

	On the other hand, similar to the proof of Lemma~\ref{l-trace}
	let $\u$ be the representative of $u$ that is defined by limits of
	integral means on the complement of some exceptional set $N$ with
	$C_{1,p}(N) = 0$ and fix $\y \in F \setminus N$. Take $W$ as in
	\eqref{d-set_quasicontinuity1} and
	\eqref{d-set_quasicontinuity2}. Repeating the arguments in the proof of
	Lemma~\ref{l-trace} reveals that the restriction of $\u$ to $\R^d
	\setminus W$ is continuous at $\y$ and that $|B(\y,r) \cap \Lambda \cap
	U \cap (\R^d \setminus W)| > 0$ if $r>0$ is small enough. By
	construction $\u = 1$ a.e.\ on $B(\y,r) \cap \Lambda \cap U \cap (\R^d
	\setminus W)$ if $r < r_0$. Hence, there is a sequence
	$\{\x_j\}_j$ approximating $\y$ such that $\u(\x_j) = 1$ for every $j
	\in \N$. By continuity $\u(\y) = 1$ follows. This proves that $\u = 1$
	holds $(1,p)$-quasieverywhere on $F$. 

	By Theorem~\ref{t-coincid} this can only happen if $C_{1,p}(F) = 0$,
	which as in (i) contradicts the $l$-thickness of $E$. \qedhere
\end{enumerate}
\end{proof}

\begin{remark} \label{r-explain}
\begin{enumerate}
 \item The proof of \ref{l-constantzero:i} shows that $\mathds 1 \notin
  W_E^{1,p}(\Lambda)$ if $E$ is merely closed and contains a relatively inner
  point that is not an inner point of $\overline{\Lambda}$.
 \item Of course the Poincar\'{e} inequality holds in the case $E =\partial
 \Lambda$ irrespective of any geometric considerations as long as $\Lambda$ is
 bounded. This can be rediscovered by the results of this section.
 Indeed, $E$ then only consists of relatively inner points and as
 $\emptyset \neq \partial \overline{\Lambda}   \subset \partial \Lambda = E$
 holds,
 it cannot be contained in the interior of $\overline{\Lambda}$. Hence
 $\mathds{1} \notin W_0^{1,p}(\Lambda)$. The  compactness of the embedding
 $W_0^{1,p}(\Lambda) \hookrightarrow  L^p(\Lambda)$ is classical and
 Theorem~\ref{t-constant} gives the claim.
\end{enumerate}
\end{remark}

Under the second assumption of Lemma~\ref{l-constantzero} there exists
a linear continuous Sobolev extension operator $\mathfrak{E}: W_E^{1,p}(\Lambda)
\to W_E^{1,p}(\R^d)$, see Theorem~\ref{t-extension always preserves trace}. Then
the compactness of the embedding $W_E^{1,p}(\Lambda) \hookrightarrow
L^p(\Lambda)$ is classical and owing to Theorem~\ref{t-constant} we can record
the following special Poincar\'{e} inequality.

\begin{proposition} \label{p-concrete_poincare}
Let $p \in {]1,\infty[}$ and let $\Lambda$ be a bounded domain. Suppose that $E
\subset \partial \Lambda$ is $l$-thick for some $l \in {]d-p,d]}$ and that for
each $\x \in \overline{\partial \Lambda \setminus E}$ there is an open
neighborhood $U_\x$ of $ \x$ such that $\Lambda \cap U_\mathrm x$ is a
$W^{1,p}$-extension domain. Then $W_E^{1,p}(\Lambda)$ may equivalently be normed
by $v \mapsto \bigl( \int_\Lambda |\nabla v|^p \mathrm{d} \x \bigr)^{1/p}$.
\end{proposition}

Now, also Theorem~\ref{t-hardy concrete} follows. In fact, this result is just
the synthesis of the above proposition with Theorems~\ref{t-hardy} and
\ref{t-extension always preserves trace}.

\section{Proof of Theorem~\ref{t-converse hardy}}\label{sec-conv-hardy}
\noindent The strategy of proof is to write $u$ as the sum of $v \in
W^{1,p}(\Omega)$ with $v/\dist_{\partial \Omega} \in L^p(\Omega)$ and $w \in
W^{1,p}$ with support within a neighborhood of $\overline{\partial \Omega
\setminus D}$. Then $v$ can be handled by the following classical result.

\begin{proposition}[{\cite[Thm.~V.3.4]{ed/ev}}]\label{p-hardy_implies_W1p0}
Let $\emptyset \subsetneq \Lambda \subsetneq \R^d$ be open and
let $p\in {]1,\infty[}$. Then if $u \in W^{1,p}(\Lambda)$ and
$u/\dist_{\partial \Lambda} \in L^p(\Lambda)$, it follows $u \in
W_0^{1,p}(\Lambda)$. 
\end{proposition}

\noindent For $w$ we can -- since local extension operators are available --
rely on the techniques developed in Section~\ref{extens}. A key observation is
an intrinsic relation between the property $\frac{u}{\dist_D} \in L^p(\Omega)$
and Sobolev regularity of the function $\log(\dist_D)$. In fact, a formal
computation gives
\begin{align*}
 \nabla (u \log(\dist_D)) = \log(\dist_D) \nabla u + \frac{u}{\dist_D} \nabla
\dist_D.
\end{align*}
Details are carried out in the following five consecutive
steps.

\subsection*{Step 1: Splitting $u$ and handling the easy term}

As in the proof of Proposition~\ref{p-extension} for
every $\x \in \overline{\partial \Omega \setminus D}$, let $U_\x$ be the open
neighborhood of $\x$ from the assumption, let $U_{\x_1}, \ldots, U_{\x_n}$ be a
finite subcovering of $\overline{\partial \Omega \setminus D}$ and let $\eps >
0$ be such that the sets $U_{\x_1},\ldots,U_{\x_n}$, together with $U
:= \{\mathrm y \in \R^d : \dist(\mathrm y, \overline{\partial \Omega
\setminus D}) > \epsilon \}$, form an open covering of $\overline \Omega$.
Finally, let $\eta, \eta_1, \ldots, \eta_n$ be a subordinated
$C_0^\infty$-partition of unity. The described splitting is $u = v + w$, where
$v:= \eta u$ and $w:= \sum_{j = 1}^n \eta_j u = (1-\eta) u$.
Since
\begin{align*}
 \dist_{\partial \Omega}(\x) \geq \min\{\eps, \dist_D(\x)\}
\geq \min\{\eps \diam(\Omega)^{-1},1 \} \cdot \dist_D(\x)
\end{align*}
holds for every $\x \in \supp(\eta) \cap \Omega$, the function $v \in
W^{1,p}(\Omega)$
satisfies
\begin{align*}
 \int_{\Omega} \left| \frac{v}{\dist_{\partial \Omega}} \right|^p \; \d \x
\leq c \int_{\Omega} \left| \frac{v}{\dist_D} \right|^p \; \d \x
\leq c \int_{\Omega} \left| \frac{u}{\dist_D} \right|^p \; \d \x < \infty
\end{align*}
by assumption on $u$. Now, Proposition~\ref{p-hardy_implies_W1p0} yields $v \in
W_0^{1,p}(\Omega) \subset W_D^{1,p}(\Omega)$.

\subsection*{Step 2: Extending $w$}

By assumption the sets $\Omega \cap U_{\x_j}$, $1 \leq j \leq n$, are
$W^{1,p}$-extension domains. Since $w =(1-\eta)u$, where $(1-\eta)$ has compact
support in the union of these domains, an extension $\widehat{w} \in
W^{1,p}(\R^d)$ of $w \in W^{1,p}(\Omega)$ with compact support within
$\bigcup_{j=1}^n U_{\x_j}$ can be constructed just as in the proof of
Proposition~\ref{p-extension}. Now, if we can show $w \in W_D^{1,p}(\Omega)$,
then by Step 1 also $u = v + w$ belongs to this space.

\subsection*{Step 3: Estimating the trace of $\widehat{w}$}

To prove $\widehat{w} \in W_D^{1,p}(\R^d)$ we rely once more on the techniques
used in the proof of Lemma~\ref{l-trace}. So, let $\widehat{\w}$ be the
representative of $\widehat{w}$ defined on $\R^d \setminus N$ via
\begin{align*}
 \widehat{\w}(\y):= \lim_{r \to 0} \frac{1}{|B(\y,r)|} \int_{B(\y, r)}
\widehat{w} \; \d \x,
\end{align*}
where the exceptional set $N$ is of vanishing $(1,p)$-capacity. Put
\begin{align*}
 U_\bigstar:= \bigcup_{j=1}^n U_{\x_j}, \quad \Omega_\bigstar := \Omega \cap
U_\bigstar, \quad \text{and} \quad D_\bigstar = \overline{D \cap U_\bigstar}
\subseteq \partial \Omega_\bigstar.
\end{align*}
Since $\widehat{w}$ has support in $U_\bigstar$ it holds $\widehat{\w}(\y) = 0$
for every $\y \in D \setminus D_\bigstar$. For the rest of the step let $\y \in
D_\bigstar \setminus N$. 

By Proposition~\ref{p-extension implies dset} each set $\Omega
\cap U_{\x_j}$ is a $d$-set and it can readily be checked that this property
inherits to their union $\Omega_\bigstar$. Hence, $\Omega_\bigstar$ satisfies
the asymptotically nonvanishing relative volume condition \eqref{e-asymp} around
$\y$ with a lower bound $c>0$ on the limes inferior that is independent of $\y$
and -- just as in the proof of Lemma~\ref{l-trace} -- a set $W \subset \R^d$ can
be constructed such that the restriction of $\widehat{\w}$ to $\R^d \setminus W$
is continuous at $\y$ and such that $|B(\y,r) \cap \Omega_\bigstar \cap (\R^d
\setminus W)| \geq cr^d/2$ if $r>0$ is small enough. By these properties of
$W$:
\begin{align*}
 |\widehat{\w}(\y)| 
&= \left|\lim_{r \to 0} \frac{1}{|B(\y,r) \cap \Omega_\bigstar \cap
(\R^d \setminus W)|} \int_{B(\y,r) \cap \Omega_\bigstar \cap (\R^d \setminus
W)} 
\widehat{\w} \; \d \x \right| \\
&\leq \limsup_{r \to 0} \frac{2}{cr^d} \int_{B(\y,r) \cap \Omega_\bigstar}
|\widehat{w}| \; \d \x \\
&= \limsup_{r \to 0} \frac{2}{cr^d} \int_{B(\y,r) \cap \Omega_\bigstar}
|w| \; \d \x.
\end{align*}
In order to force these mean-value integral to vanish in the limit $r \to 0$,
introduce the function $\log(\dist_D)^{-1}$, which is bounded above in absolute
value by $|\log r|^{-1}$ on $B(\y,r)$ if $r < 1$. It follows
\begin{align} \label{e-estimate_tracewidehat}
|\widehat{\w}(\y)| \leq c \limsup_{r \to 0} |\log r|^{-1} 
\bigg(\frac{1}{r^d} \int_{B(\y,r) \cap \Omega_\bigstar} 
|w \log(\dist_D)| \; \d \x \bigg).
\end{align}
So, since $|\log r|^{-1} \to 0$ as $r \to 0$ the function $\widehat{\w}$
vanishes at every $\y \in D_\bigstar \setminus N$ for which the mean value
integrals on the right-hand side remain bounded as $r$ tends to zero.

\subsection*{Step 4: Intermezzo on $w \log(\dist_D)$}

In this step we prove the following result.

\begin{lemma}
\label{Lem: Logdist lemma}
Let $p \in {]1,\infty[}$, let $\Lambda \subset \R^d$ be a bounded $d$-set, and
let $E \subset \partial \Lambda$ be closed and porous. Suppose $u \in
W^{1,p}(\Lambda)$ has an extension $v \in W^{1,p}(\R^d)$ and satisfies
$\frac{u}{\dist_E} \in L^p(\Lambda)$. If $r \in {]1,p[}$ and $s \in {]0,1[}$,
then the function $|u \log(\dist_E)|$ defined on
$\Lambda$ has an extension in the Bessel potential space $H^{s,r}(\R^d)$ that
is positive almost everywhere.
\end{lemma}

For the proof we need the following extension result of Jonsson and Wallin.

\begin{proposition}[{\cite[Thm.~V.1.1]{jons}}]
\label{p-Jonsson-Wallin theorem on dsets}
Let $s \in {]0,1[}$, $p \in {]1,\infty[}$ and let $\Lambda \subset \R^d$ be a
$d$-set. Then there exists a linear operator $\mathfrak{E}$ that extends every
measurable function $f$ on $\Lambda$ that satisfies
\begin{align*}
 \|f\|_{L^p(\Lambda)} + \bigg(\iint_{\substack{\x,\y \in
\Lambda \\ |\x-\y| < 1}}  \frac{|f(\x) - f(\y)|^p}{|\x-\y|^{d+sp}} \;
\d \x \; \d \y \bigg)^{1/p} < \infty
\end{align*}
to a function $\mathfrak{E} f$ in the Besov space $B_s^{p,p}(\R^d)$ of all
measurable functions $g$ on $\R^d$ such that
\begin{align*}
 \|g\|_{L^p(\R^d)} + \bigg(\iint_{\x,\y \in \R^d}  \frac{|g(\x) -
g(\y)|^p}{|\x-\y|^{d+sp}} \;
\d \x \; \d \y \bigg)^{1/p} < \infty.
\end{align*}
\end{proposition}

\begin{remark} \label{r-Jonsson-Wallin theorem on dsets}
The Besov spaces are nested with the Bessel potential spaces in the sense that
$B_s^{p,p}(\R^d) \subset H^{s-\eps,p}(\R^d)$ for each $s>0$ and every $\eps \in
{]0,s[}$.
Moreover, $W^{1,p}(\R^d) \subset B_s^{p,p}(\R^d)$. Proofs of these results
can be found e.g.\ in \cite[Sec.~2.3.2/2.5.1]{triebel}.
\end{remark}

\begin{proof}[Proof of Lemma~\ref{Lem: Logdist lemma}]
Using Remark~\ref{r-Jonsson-Wallin theorem on dsets} it suffices to construct
an extension in $B_s^{p,p}$ with the respective properties. Moreover, by the
reverse triangle inequality it is enough to construct any extension $f \in
B_s^{p,p}(\R^d)$ of $u \log \dist_E$ -- then $|f|$ can be used as the required
extension of $|u \log \dist_E|$. These considerations and
Proposition~\ref{p-Jonsson-Wallin theorem on dsets} show that the claim follows
provided
\begin{align}
\label{Eq: Wsr norm of logdist function}
 \|u \log(\dist_D)\|_{L^r(\Lambda)} + \bigg(\iint_{\substack{x,y \in \Lambda \\
|\x-\y| < 1}}  \frac{|u(\x)\log(\dist_E(\x))-
u(\y)\log(\dist_E(\y))|^r}{|\x-\y|^{d+sr}} \; \d \x \; \d \y \bigg)^{1/r}
\end{align}
is finite.

To bound the $L^r$ norm on the left-hand side of \eqref{Eq: Wsr norm of logdist
function} choose $q \in {]1,\infty[}$ such that $\frac{1}{r} = \frac{1}{p} +
\frac{1}{q}$ and apply H\"older's inequality
\begin{align*}
 \|u \log(\dist_E)\|_{L^r(\Lambda)} \leq \|u\|_{L^p(\Lambda)}
\|\log(\dist_D)\|_{L^q(\Lambda)}.
\end{align*}
For the second term on the right-hand we use that the Aikawa dimension of the
porous set $E$ is strictly smaller than $d$, see Remark~\ref{r-porous}. This
entails for some $\alpha < d$ and some $\x \in E$ the estimate
\begin{align*}
 \int_\Lambda \dist_E(\x)^{\alpha-d} \; \d \x \leq \int_{B(\x,2 \diam \Lambda)}
\dist_E(\x)^{\alpha-d} \; \d \x \leq c_\alpha (2 \diam \Lambda)^\alpha < \infty.
\end{align*}
Hence, some negative power of $\dist_E$ is integrable on $\Lambda$ and by
subordination of logarithmic growth $\log(\dist_E) \in L^q(\Lambda)$ follows.
Altogether, $u \log(\dist_E) \in L^r(\Lambda)$ taking care of the first term in
\eqref{Eq: Wsr norm of logdist function}.

By symmetry the domain of integration for the second term on the left-hand side
of \eqref{Eq: Wsr norm of logdist function} can be restricted to $\dist_E(\x)
\geq \dist_E(\y)$. By adding and subtracting the term $u(\y) \log(\dist_E(\x))$
it in fact suffices to prove that
\begin{align}
\label{Eq1: Wsr seminorm logdist function}
\bigg(\int_\Lambda \int_\Lambda \frac{|u(\x)-u(\y)|^r}{|\x-\y|^{d+sr}}
|\log(\dist_E(\x))|^r \; \d \x \; \d \y \bigg)^{1/r}
\end{align}
and
\begin{align}
\label{Eq2: Wsr seminorm logdist function}
\bigg(\int_\Lambda |u(\y)|^r \int_{\substack{\x \in \Lambda \\ \dist_E(\x) \geq
\dist_E(\y)}} \frac{|\log(\dist_E(\x))-\log(\dist_E(\y))|^r}{|\x-\y|^{d+sr}}
\; \d \x \; \d \y\bigg)^{1/r}
\end{align}
are finite. Fix $s< t< 1$, write \eqref{Eq1: Wsr seminorm logdist function} in
the form
\begin{align*}
&\bigg(\int_\Lambda \int_\Lambda \frac{|u(\x)-u(\y)|^r}{|\x-\y|^{dr/p+tr}}
\frac{|\log(\dist_E(\x))|^r}{|\x-\y|^{dr/q + sr - tr}} \; \d \x \; \d \y
\bigg)^{1/r}
\intertext{and apply H\"older's inequality with $\frac{1}{r} = \frac{1}{p} +
\frac{1}{q}$ to bound it by}
&\leq \bigg(\int_\Lambda \int_\Lambda \frac{|u(\x)-u(\y)|^p}{|\x-\y|^{d+tp}}\;
\d \x \; \d \y \bigg)^{1/p}
 \bigg(\int_\Lambda \int_\Lambda
\frac{|\log(\dist_E(\x))|^q}{|\x-\y|^{d+(s-t)q}} \; \d \y \; \d \x
\bigg)^{1/q}\\
&\leq \|\log(\dist_E)\|_{L^q(\Lambda)} \bigg(\int_\Lambda \int_\Lambda
\frac{|u(\x)-u(\y)|^p}{|\x-\y|^{d+tp}}\; \d \x \; \d \y \bigg)^{1/p}
\bigg(\int_{|\y| \leq \diam(\Lambda)} \frac{1}{|\y|^{d+(s-t)q}} \; \d
\y \bigg)^{1/q}
\end{align*}
Now, $\log(\dist_E) \in L^q(\Lambda)$ has been proved above and the third
integral is absolutely convergent since $d+(s-t)q < d$. Finally note that by
assumption $u$ has an extension $v \in W^{1,p}(\R^d)$. Since $W^{1,p}(\R^d)
\subset B_s^{p,p}(\R^d)$ the middle term above is finite as well, see
Remark~\ref{r-Jonsson-Wallin theorem on dsets}.

It remains to show that the most interesting term \eqref{Eq2: Wsr seminorm
logdist function} is finite. Here, the additional assumptions on $u$, $s$ and
$r$ enter the game. By the mean value theorem for the logarithm and since
$\dist_E$ is a contraction, the $r$-th power of this term is bounded above by
\begin{align*}
&\int_\Lambda |u(\y)|^r \int_{\substack{\x \in \Lambda \\ \dist_E(\x) \geq
\dist_E(\y)}} \frac{|\dist_E(\x) - \dist_E(\y)|^r}{\dist_E(\y)^r|\x-\y|^{d+sr}}
\; \d \x \; \d \y \\
\leq &\int_\Lambda \bigg|\frac{u(\y)}{\dist_E(\y)} \bigg|^r \int_{\Lambda}
\frac{|\x-\y|^r}{|\x-\y|^{d+sr}} \; \d \x \; \d \y\\
\leq  &\int_\Lambda \bigg|\frac{u(\y)}{\dist_E(\y)} \bigg|^r \; \d y
\int_{|\x| \leq \diam(\Lambda)} \frac{1}{|\x|^{d+r(s-1)}} \; \d \x.
\end{align*}
Now, the integral with respect to $\x$ is finite since $r(s-1) < 0$. The
integral with respect to $\y$ is finite since by assumption
$\frac{u}{\dist_E}$ is $p$-integrable on the bounded domain $\Lambda$ and thus
$r$-integrable for every $r<p$.
\end{proof}

On noting that by Definition~\ref{d-porous} a subset of a porous set is again
porous, Lemma~\ref{Lem: Logdist lemma} applies to the bounded $d$-set
$\Omega_\bigstar$ and the porous set $D_\bigstar \subset D$. Moreover, $w = (1-
\eta)u \in W^{1,p}(\Omega_\bigstar)$ has the extension $\widehat{w} \in
W^{1,p}(\R^d)$ and satisfies
\begin{align*}
 \int_{\Omega_\bigstar} \bigg| \frac{w(\x)}{\dist_{D_\bigstar}(\x)} \bigg|^p \;
\d \x
\leq \|1-\eta\|_\infty \int_{\Omega} \bigg| \frac{u(\x)}{\dist_{D}(\x)}
\bigg|^p < \infty.
\end{align*}
Hence we can record:

\begin{corollary}\label{c-extension wlogdist}
For every $r \in {]1,p[}$ and every $s \in {]0,1[}$ the
function $|w \log(\dist_{D_\bigstar})|$ defined on $\Omega_\bigstar$
has an extension $f_{s,r} \in H^{s,r}(\R^d)$ that is positive almost everywhere.
\end{corollary}

\subsection*{Step 5: Re-inspecting the right-hand side of
\eqref{e-estimate_tracewidehat}}

We return to \eqref{e-estimate_tracewidehat}. Given $r \in {]1,p[}$ and $s
\in {]0,1[}$ let $f_{s,r} \in H^{s,r}(\R^d)$ be as in
Corollary~\ref{c-extension wlogdist}. By \eqref{e-Lebesgue} we can infer
\begin{align*}
 \limsup_{r \to 0} \frac{1}{r^d} \int_{B(\y,r) \cap \Omega_\bigstar} |w
\log(\dist_D)| \; \d \x
\leq \limsup_{r \to 0} \frac{1}{r^d} \int_{B(\y,r)} f_{s,r} \; \d \x
<
\infty
\end{align*}
for $(s,r)$-quasievery $\y \in D_\bigstar \setminus N$. By the conclusion of
Step~3 this implies $\widehat{\w}(\y) = 0$ for $(s,r)$-quasievery $\y \in
D_\bigstar \setminus N$. To proceed further, we distinguish two cases:

\begin{enumerate}
 \item It holds $p \leq d$. Since the product $sr < p\leq d$ can get
 arbitrarily close to $p$, Lemma~\ref{l-transformation of capacities}
 yields for every $r \in {]1,p[}$ that $\widehat{\w} = 0$ holds
 $(1,r)$-quasieverywhere on $D_\bigstar \setminus N$. Moreover, since
 $C_{1,p}(N) = 0$ by definition, $\widehat{\w} = 0$ holds even
 $(1,r)$-quasieverywhere on $D_\bigstar$.

 \item It holds $p > d$. Then $\widehat{\w}$ is the continuous
 representative of $\widehat{w} \in W^{1,p}(\R^d)$ and $N$ is empty, see the
 beginning of Step~3. Moreover, we can choose $s$ and $r$ such that $d-l < sr$
 and conclude from the comparison theorem, Theorem~\ref{t-comparison}, that
 $\widehat{\w}$ vanishes $\H_l^\infty$-a.e.\ on $D_\bigstar$. Since $D$ is
 $l$-thick and $U_\bigstar$ is open, for each $\y \in D \cap U_\bigstar$ the
 set $B(\y,r) \cap D \cap U_\bigstar$ coincides with $B(\y,r) \cap D$ provided
 $r>0$ is small enough and thus has strictly positive $\H_l^\infty$-measure.
 So, the continuous function $\widehat{\w}$ has to vanish \emph{everywhere} on
 $D \cap U_\bigstar$ as well as on the closure of the latter set -- which by
 definition is $D_\bigstar$.
\end{enumerate}

\noindent Summing up, $\widehat{\w} = 0$ has been shown to hold
$(1,r)$-quasieverywhere on $D_\bigstar$ for every $r \in {]1,p[}$. From the
beginning of Step~3 we also know that $\widehat{\w}$ vanishes everywhere on
$D \setminus D_\bigstar$ and as $\widehat{w} \in W^{1,p}(\R^d)$ has compact
support, H\"older's inequality yields $\widehat{w} \in W^{1,r}(\R^d)$. Combining
these two observations with Theorem~\ref{t-coincid} we are eventually led to
\begin{align}
\label{e-final observation converse hardy}
 \widehat{w} \in W^{1,p}(\R^d) \cap \bigcap_{1<r<p} W_D^{1,r}(\R^d).
\end{align}
We continue by quoting the following result of Hedberg and Kilpel\"ainen.

\begin{proposition}[{\cite[Cor.~3.5]{hedberg-stability}}]\label{p-stability}
Let $p \in {]1,\infty[}$ and let $\Lambda \subset \R^d$ be a bounded domain
whose boundary is $l$-thick for some $l \in {]d-p,d]}$. Then
\begin{align*}
 W^{1,p}(\Lambda) \cap \bigcap_{1<r<p} W_0^{1,r}(\Lambda) \subset
W_0^{1,p}(\Lambda).
\end{align*}
\end{proposition}

\begin{remark}\label{r-stability}
In \cite{hedberg-stability} the requirement on $\Lambda$ is that its complement
is uniformly $p$-fat -- a property that by the ingenious characterization in
\cite[Thm.~1]{juha} holds for every bounded set with $l$-thick boundary provided
$l \in {]d-p,d]}$. 
\end{remark}

In order to apply this result to the case of mixed boundary conditions, we
proceed similarly to the proof of Theorem~\ref{t-hardy}: With $B \subset \R^d$ 
an open ball that contains the compact support of $\widehat{w}$ define again
\begin{align*}
 \mathcal C := \{ M \subset B \setminus D : M \text{ open, connected
	and } \Omega \subset M \}
\end{align*}
and
\begin{align*}
 \Omega_\bullet :=
	\bigcup_{M \in \mathcal C} M.
\end{align*}
Then $\partial \Omega_\bullet \in \{D, D \cup \partial B\}$ by
Corollary~\ref{c-002HE}, subsequent to which it is also shown that $\partial
\Omega_\bullet$ is $m$-thick for some $m \in {]d-p,d]}$. Finally, let $\eta \in
C_0^\infty(B)$ be identically one on the support of $\widehat{w}$. As
$\varphi \mapsto (\eta \varphi)|_{\Omega_\bullet}$ induces a bounded
operator $W_D^{1,p}(\R^d) \to W_0^{1,p}(\Omega_\bullet)$, it follows from
\eqref{e-final observation converse hardy} that
\begin{align*}
  \widehat{w}|_{\Omega_\bullet} = (\eta \widehat{w})|_{\Omega_\bullet} \in
W^{1,p}(\Omega_\bullet) \cap \bigcap_{1<r<p} W_0^{1,r}(\Omega_\bullet)
\end{align*}
and thus $\widehat{w}|_{\Omega_\bullet} \in W_0^{1,p}(\Omega_\bullet)$ thanks to
Proposition~\ref{p-stability}. Since by construction $\Omega \subset
\Omega_\bullet$  and $D \subset \partial \Omega_\bullet$, we eventually conclude
\begin{align*}
 w = \widehat{w}|_\Omega \in W_D^{1,p}(\Omega)
\end{align*}
and the proof is complete. $\hfill \square$

\section{A Generalization} \label{sec-general}
\noindent If one asks: `What is the most restricting condition in
Theorem~\ref{t-hardy}?', the answer doubtlessly is the assumption that a
\emph{global} extension operator shall exist. Certainly, this excludes all
geometries that include cracks not belonging completely to the Dirichlet
boundary part as in the subsequent Figure.
\begin{figure}[htbp]
\centerline{\includegraphics[scale=0.4]{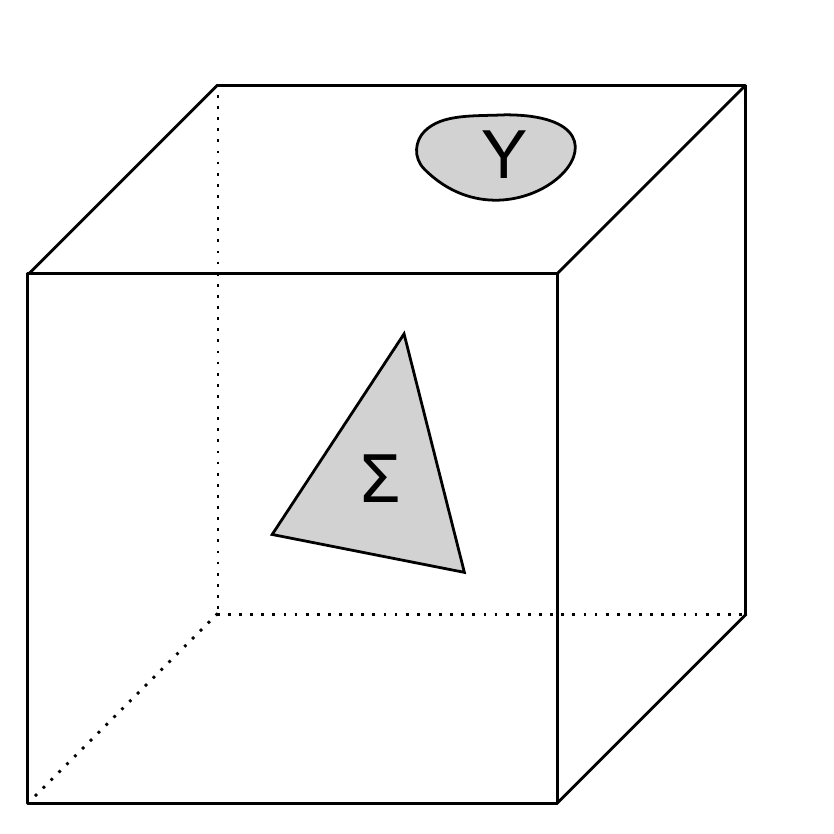}}
\caption{\label{fig-fehlendes_Dreieck} 
The domain $\Omega$ is the cube minus
 the triangle $\Sigma$. The Dirichlet boundary part $D$ consists exactly of the
 six outer sides of the cube minus the droplet $\Upsilon$ on the cover plate.}
\end{figure}

Since the distance function $\dist_D$ measures only the distance to the
Dirichlet boundary part $D$, points in $\partial \Omega$ that are far from $D$
should not be of great relevance in view of the Hardy
inequality~\eqref{e-hardy0}. In the following considerations we intend to make
this precise. Let $U,V \subset \R^d$ be two open, bounded sets with the
properties 
\begin{equation} \label{e-propUV}
 D \subset U, \quad \overline V \cap D = \emptyset, \quad \overline \Omega
	\subset U \cup V.
\end{equation}
The philosophy behind this is to take $U$ as a small neighborhood of $D$ which
-- desirably -- excludes the `nasty parts' of $\partial \Omega \setminus D$.
More properties of $U,V$ will be specified below. Let $\eta_U \in
C^\infty_0(U), \eta_V \in C^\infty_0(V)$ be two functions with $\eta_U + \eta_V
= 1$ on $\overline \Omega$. Then one can estimate
\[ \Bigl( \int_\Omega |u|^p \dist_D^{-p} \; \mathrm{d} \x \Bigr)^{1/p}
	\le \Bigl( \int_{U \cap \Omega} |\eta_U u|^p \dist_D^{-p}
	\; \mathrm{d} \x \Bigr)^{1/p} + \Bigl( \int_{V \cap \Omega}
	|\eta_V u|^p \dist_D^{-p} \; \mathrm{d} \x \Bigr)^{1/p}.
\]
Since $\dist_D$ is larger than some $\epsilon > 0$ on $\mathrm{supp}(\eta_V)
\subset V$, the second term can be estimated by $\frac{1}{\epsilon} \bigl(
\int_{\Omega} |u|^p \; \mathrm{d} \x \bigr)^{1/p}$. If one assumes, as
above, Poincar\'e's inequality, then this term may be estimated as required. In
order to provide an adequate estimate also for the first term, we introduce the
following assumption.

\begin{assumption} \label{a-1}
 The set $U$ from above can be chosen in such a way that $\Lambda := \Omega
 \cap U$ is again a domain and if one puts $\Gamma := (\partial \Omega
 \setminus D) \cap U$ and $E := \partial \Lambda \setminus \Gamma$, then there
 is a linear, continuous extension operator $\mathfrak F : W^{1,p}_E(\Lambda)
 \to W^{1,p}_E(\R^d)$.
\end{assumption}
Clearly, this assumption is weaker than Condition~\ref{t-hardy:iii} in
Theorem~\ref{t-hardy}; in other words: Condition~\ref{t-hardy:iii} in
Theorem~\ref{t-hardy} requires Assumption~\ref{a-1} to hold for an open set $U
\supset \overline \Omega$.

Let us discuss the sense of Assumption~\ref{a-1} in extenso. Philosophically
spoken, it allows to focus on the extension not of the functions $u$ but the
functions $\eta_U u$, which live on a set whose boundary does (possibly) not
include the `nasty parts' of $\partial \Omega \setminus D$ that are an
obstruction against a global extension operator. In detail: one first observes
that, for $\eta = \eta_U \in C^\infty_0(U)$ and $v \in W^{1,p}_D(\Omega)$, the
function $\eta v|_\Lambda$ even belongs to $W^{1,p}_E(\Lambda)$, see
\cite[Thm. 5.8]{haller}. Secondly, we have by the definition of $E$
\[ \partial U \cap \Omega = (\partial U \cap \Omega) \setminus \Gamma \subset
	\partial \Lambda \setminus \Gamma = E.
\]
This shows that the `new' boundary part $\partial U \cap \Omega$ of $\Lambda$
belongs to $E$ and is, therefore, uncritical in view of extension. Thirdly,
one has $D = D \cap U \subseteq \partial \Omega \cap U \subset
\partial \Lambda$, and it is clear that the `new Dirichlet boundary part' $E$
includes the `old' one $D$. Hence, the extension operator $\mathfrak F$ may be
viewed also as a continuous one between $W^{1,p}_E(\Lambda)$ and
$W^{1,p}_D(\R^d)$. Thus, concerning $v := \eta u = \eta_U u$ one is -- mutatis
mutandis -- again in the situation of Theorem~\ref{t-hardy}: $\eta u \in
W^{1,p}_E(\Lambda) \subset W^{1,p}_D(\Lambda)$ admits an extension
$\mathfrak F(\eta u) \in W^{1,p}_E(\R^d) \subseteq W^{1,p}_D(\R^d)$, which 
satisfies the estimate $\|\mathfrak F(\eta u)\|_{W^{1,p}_D(\R^d)} \le c \|\eta
u\|_{W^{1,p}_D(\Lambda)}$, the constant $c$ being independent from $u$. This
leads, as above, to a corresponding (continuous) extension operator
$\mathfrak F_\bullet : W^{1,p}_E(\Lambda) \to W^{1,p}_0(\Lambda_\bullet)$.
Here, of course, $\Lambda_\bullet$ has again to be defined as the connected
component of $B \setminus D$ that contains $\Lambda$. Thus one may proceed
again as in the proof of Theorem \ref{t-hardy}, and gets, for $u \in
W^{1,p}_D(\Omega)$,
\begin{align*}
  \int_\Omega \Bigl( \frac{|\eta u|}{\dist_D} \Bigr)^p \; \mathrm{d} \x
	&= \int_\Lambda \Bigl( \frac{|\eta u|}{\dist_D} \Bigr)^p
	\; \mathrm{d} \x \le \int_{\Lambda_\bullet} \Bigl(
	\frac{|\mathfrak F_\bullet (\eta u)|}{\dist_{\partial \Lambda_\bullet}}
	\Bigr)^p \; \mathrm{d} \x \le c \|\nabla (\mathfrak F_\bullet
	(\eta u))\|^p_{L^p(\Lambda_\bullet)} \\
  &\le c \|\mathfrak F_\bullet (\eta u)\|^p_{W^{1,p}(\Lambda_\bullet)} \le c
	\|\eta u\|^p_{W^{1,p}(\Lambda)} \le c \bigl( \|u\|_{L^p(\Omega)}^p +
	\|\nabla u\|_{L^p(\Omega)}^p \bigr),
\end{align*}
since $\eta u$ belongs to $W^{1,p}_E(\Lambda) \subset W^{1,p}_D(\Lambda)$.
Exploiting a last time Poincar\'e's inequality, whose validity will be
discussed in a moment, one gets the desired estimate.

When aiming at Poincar\'e's inequality, it seems convenient to follow again the
argument in the proof of Proposition~\ref{t-constant}: as pointed out above, the
property $\mathds{1} \notin W^{1,p}_D(\Omega)$ has to do only with the local
behavior of $\Omega$ around the points of $D$,
cf.\@ Lemma~\ref{l-constantzero}. Hence, this will not be discussed further
here.

Concerning the compactness of the embedding $W^{1,p}_D(\Omega) \hookrightarrow
L^p(\Omega)$, one does not need the existence of a \emph{global} extension
operator $\mathfrak E : W^{1,p}_D(\Omega) \to W^{1,p}(\R^d)$. In fact, writing
for every $v \in W^{1,p}_D(\Omega)$ again $v = \eta_U v + \eta_V v$ and
supposing Assumption~\ref{a-1}, one gets the following: 

If $\{v_k\}_k$ is a bounded sequence in $W^{1,p}_D(\Omega)$, then the sequence
$\{\eta_U v_k|_\Lambda\}_k$ is bounded in $W^{1,p}_E(\Lambda)$. Due to the
extendability property, this sequence contains a subsequence $\{\eta_U
v_{k_l}|_\Lambda\}_l$ that converges in $L^p(\Lambda)$ to an element $v_U$.
Thus, $\{\eta_U v_{k_l}\}_l$ converges to the function on $\Omega$ that equals
$v_U$ on $\Lambda$ and $0$ on $\Omega \setminus \Lambda$. The elements $\eta_V
v_k$ in fact live on the set $\Pi := \Omega \cap V$ and
are zero on $\Omega \setminus V$. In particular they are zero in a
neighborhood of $D$. Moreover, they form a bounded subset of $W^{1,p}(\Pi)$.
Therefore it makes sense to require that $\Pi$ is again a domain, and,
secondly that $\Pi$ meets one of the well-known compactness criteria
$W^{1,p}(\Pi) \hookrightarrow L^p(\Pi)$, cf.\@ \cite[Ch.~1.4.6]{mazsob}. Keep in
mind that such requirements are much weaker than the global
$W^{1,p}$-extendability, and in particular include the example in
Figure~\ref{fig-fehlendes_Dreieck}, as long as the triangle $\Sigma$ has a
positive distance to the six outer sides of the cube. Resting on these criteria,
one obtains again the convergence of a subsequence $\{\eta_V v_{k_l}|_{\Pi}\}_l$
that converges in $L^p(\Pi)$ towards a function $v_V$. The sequence
$\{\eta_V v_{k_l}\}_l$ then converges in $L^p(\Omega)$ to a function that equals
$v_V$ on $\Pi$ and zero on $\Omega \setminus V$.

Altogether, we have extracted a subsequence of $\{v_k\}_k$ that converges in
$L^p(\Omega)$.

\begin{remark} \label{r-zusammh}
 In fact one does not really need that $\Pi$ is connected. By similar
 arguments as above it suffices to demand that it splits up in at most finitely
 many components $\Pi_1, \ldots, \Pi_n$, such that each of these admits the
 compactness of the embedding $W^{1,p}(\Pi_j) \hookrightarrow L^p(\Pi_j)$.
\end{remark}

We summarize these considerations in the following theorem.

\begin{theorem} \label{t-hardy1}
 Let $\Omega \subset \R^d$ be a bounded domain and $D \subset \partial \Omega$
 be a closed part of the boundary. Suppose that the following three conditions
 are satisfied:
 \begin{enumerate}
  \item The set $D$ is $l$-thick for some $l \in {]d-p,d]}$.
  \item The space $W^{1,p}_D(\Omega)$ can be equivalently normed by
	$\|\nabla \cdot \|_{L^p(\Omega)}$.
  \item There are two open sets $U, V \subset \R^d$ that satisfy
	\eqref{e-propUV} and $U$ satisfies Assumption~\ref{a-1}.
 \end{enumerate}
 Then there is a constant $c>0$ such that Hardy's inequality
 \begin{equation*}
  \int_\Omega \Bigl| \frac{u}{\dist_D} \Bigr|^p \; \dd \mathrm{x}  \le c
	\int_\Omega |\nabla u|^p \; \dd \mathrm{x}
 \end{equation*}
 holds for all $u \in W^{1,p}_D(\Omega)$.
\end{theorem}

\subsection*{Acknowledgment} 
 In 2012, after we asked him a question, V.~Maz'ya proposed a
  proof \cite{mazcomm} of Theorem \ref{t-hardy} in the case of $D$ a
  $(d-1)$-set that heavily relied on 
  several deep results from his book \cite{mazsob}. Actually, we found, after
  this, references in the literature \cite{korte}, \cite{juha} with a different
  approach that apply directly, provided Lemma~\ref{l-lsets with hausdorff
  content} is established. It was again V.~Maz'ya who drew
  our attention to the fact that something like this lemma is needed.
  We warmly thank him for all that.

  We also thank the anonymous referee who has made many excellent suggestions
  allowing to considerably improve a previous version of this
manuscript.


\end{document}